\documentclass[a4paper,12pt,reqno]{amsart} 
\usepackage{amsmath,amsfonts,amssymb,amsthm}
\usepackage[margin=1.5cm]{geometry}
\usepackage{esint} 
\PassOptionsToPackage{dvipsnames}{xcolor}
\usepackage{tikz}

\usepackage{color}

\newcommand*\Laplace{\mathop{}\!\mathbin\bigtriangleup}

\newcommand{\shom}{\textrm{hom}}  

\providecommand{\dx}{\, \mathrm{d} x}
\providecommand{\dy}{\, \mathrm{d} y}
\providecommand{\dz}{\, \mathrm{d} z}
\newcommand{\en}[1]{\left< #1 \right>}  
  
\providecommand{\p}[1]{\left(#1\right)}
\newcommand{\supp}{\mathop{\mathrm{supp}}}
\renewcommand{\p}[1]{\left( #1 \right)}

\providecommand{\R}{\mathbb{R}}
\providecommand{\N}{\mathbb{N}}

\newcommand{\step}[1]{\noindent \textit{Step} #1.}

\newtheorem{theorem}{Theorem}
\newtheorem{lemma}[theorem]{Lemma}
\newtheorem{definition}[theorem]{Definition}
\newtheorem{proposition}[theorem]{Proposition}

\begin{document}

\title[Second-order correctors in random homogenization]{Stochastic Homogenization of Linear Elliptic Equations: Higher-Order Error Estimates in Weak Norms via Second-Order Correctors}
\author{Peter Bella}
\address{Institute of Mathematics,
University of Leipzig, Augustusplatz 10, 04109 Leipzig, Germany}\email{peter.bella@mis.mpg.de}
\author{Benjamin Fehrman}
\address{Max Planck Institute for Mathematics in the Sciences, Inselstrasse 22, 04103 Leipzig, Germany}\email{benjamin.fehrman@mis.mpg.de}
\author{Julian Fischer}
\address{Max Planck Institute for Mathematics in the Sciences, Inselstrasse 22, 04103 Leipzig, Germany}
\email{julian.fischer@mis.mpg.de}
\author{Felix Otto}
\address{Max Planck Institute for Mathematics in the Sciences, Inselstrasse 22, 04103 Leipzig, Germany}
\email{otto@mis.mpg.de}

\begin{abstract}
We are concerned with the homogenization of second-order linear elliptic equations with random coefficient fields. For symmetric coefficient fields with only short-range correlations, quantified through a logarithmic Sobolev inequality for the ensemble, we prove that when measured in weak spatial norms, the solution to the homogenized equation provides a higher-order approximation of the solution to the equation with oscillating coefficients. In the case of nonsymmetric coefficient fields, we provide a higher-order approximation (in weak spatial norms) of the solution to the equation with oscillating coefficients in terms of solutions to constant-coefficient equations. In both settings, we also provide optimal error estimates for the two-scale expansion truncated at second order. Our results rely on novel estimates on the second-order homogenization corrector, which we establish via sensitivity estimates for the second-order corrector and a large-scale $L^p$ theory for elliptic equations with random coefficients. Our results also cover the case of elliptic systems.
\end{abstract}

\thanks{The first author is supported by the DFG grant BE 5922/1.1. The second author is supported by the National Science Foundation Mathematical Sciences Postdoctoral Research Fellowship under Grant Number 1502731}

\maketitle

\section{Introduction}

In the present work, we study the homogenization of linear elliptic equations of the form
\begin{align*}
-\nabla \cdot (a\nabla u) = f.
\end{align*}
According to the quantitative theory of stochastic homogenization, for \emph{random coefficient fields} with typical length of correlation $\varepsilon$ on $\mathbb{R}^d$, one may approximate the solution $u$ to this problem with microstructure by the solution $u_{\shom}$ of a constant-coefficient \emph{effective equation}
\begin{align*}
-\nabla \cdot (a_{\shom}\nabla u_{\shom}) = f
\end{align*}
up to an average error of the order of $\varepsilon$ (in case of three or more spatial dimensions and smooth compactly supported $f$). More precisely, for coefficient fields with finite range of dependence, one obtains an error estimate of the form
\begin{align}
\label{ErrorEstimateuhom}
\left(\fint_{\{|x|\leq 1\}}|u-u_{\shom}|^2 \,dx\right)^{1/2} \leq \mathcal{C}(a) \varepsilon
\end{align}
with a random constant $\mathcal{C}(a)$ subject to an estimate of the form
\begin{align*}
\mathbb{E}[\exp(\mathcal{C}(a)^{2})]\leq C(f).
\end{align*}
Note that the scaling of the estimate \eqref{ErrorEstimateuhom} with respect to $\varepsilon$ is optimal, as can for example be seen by our $H^1$-norm error estimate for the two-scale expansion truncated at second order in Theorem~\ref{ErrorEstimateSymmetric} below.

The main goal of the present paper is to establish a \emph{higher-order approximation result} for solutions to the microscopic problem (with respect to the typical length of correlations $\varepsilon$) in terms of solutions to the effective equation, as compared to the estimate \eqref{ErrorEstimateuhom}. In the case of three or four spatial dimensions $d=3$ or $d=4$ and for \emph{symmetric} coefficient fields $a$, we shall prove that in a weaker spatial norm the approximation $u\approx u_{\shom}$ actually achieves an average error approximately of the order of $\varepsilon^{d/2}$: We shall establish the bound
\begin{align}
\label{ImprovedEstimate}
\big|\big| u-u_{\shom} \big|\big|_{H^{-1}(\{|x|\leq 1\})}
\leq
\begin{cases}
\mathcal{C}(a) \varepsilon^{3/2}&\text{for }d=3,
\\
\mathcal{C}(a) \varepsilon^2 |\log \varepsilon| &\text{for }d=4,
\\
\mathcal{C}(a) \varepsilon^2 &\text{for }d\geq 5,
\end{cases}
\end{align}
with a random constant $\mathcal{C}(a)$ having stretched exponential moments in the sense
\begin{align*}
\mathbb{E}[\exp(\mathcal{C}(a)^{1/C})]\leq C
\end{align*}
for some $C>0$. Note that while the dependence of our error bound on $\varepsilon$ is (almost) optimal, we do not expect the moment bound on the random constant $\mathcal{C}(a)$ to be optimal.

Approximation results of this type are of particular interest in the case of stochastic homogenization: In the periodic homogenization of elliptic PDEs, a higher-order approximation of the solution $u$ to the problem with oscillating coefficients is easily obtained even in $L^p$ spaces by the two-scale expansion (see \eqref{TwoScaleExpansion} below); note that in the periodic setting the corrector may be computed cheaply by exploiting its periodicity. In contrast, in stochastic homogenization the computation of the corrector is as expensive as computing the solution to the equation with oscillating coefficients itself, making a computation of the corrector infeasible.

If the coefficient field $a$ is non-symmetric, we also obtain a higher-order approximation of the solution to the microscopic problem in terms of solutions to macroscopic problems (that is, in terms of constant-coefficient equations), however at the expense of having to solve an additional macroscopic equation: We obtain the approximation results
\begin{align*}
\big|\big|u-u_{\shom}-\varepsilon u_{\shom}^1 \big|\big|_{H^{-1}(\{|x|\leq 1\})}
\leq
\begin{cases}
\mathcal{C}(a) \varepsilon^{3/2}&\text{for }d=3,
\\
\mathcal{C}(a) \varepsilon^2 |\log \varepsilon| &\text{for }d=4,
\\
\mathcal{C}(a) \varepsilon^2 &\text{for }d\geq 5,
\end{cases}
\end{align*}
where $u_{\shom}^1$ is the solution to a constant-coefficient equation of the form
\begin{align}
\label{EffectiveEquationFirst}
-\nabla \cdot (a_{\shom}\nabla u_{\shom}^1) = \sum_{i,j,k=1}^d a_{\shom,ijk}^1 \partial_i \partial_j \partial_k u_{\shom}.
\end{align}
Here, the constants $a_{\shom,ijk}^1$ depend only on the statistics of the random coefficient field and are given by the formula \eqref{Defa1ijk} below.

Our results may be regarded as a quantitative counterpart of the recent work of Gu \cite{YuGu}, who establishes that higher-order homogenization correctors allow for a qualitatively better approximation of solutions to the equation with random coefficients:
At the level of the $L^2$-norm, for symmetric coefficient fields we obtain the approximation result
\begin{align*}
\Big|\Big|u-u_{\shom}-\sum_i \phi_i \partial_i u_{\shom}\Big|\Big|_{L^2(\{|x|\leq 1\})}
\leq
\begin{cases}
\mathcal{C}(a) \varepsilon^{3/2}&\text{for }d=3,
\\
\mathcal{C}(a) \varepsilon^2 |\log \varepsilon| &\text{for }d=4,
\\
\mathcal{C}(a) \varepsilon^2 &\text{for }d\geq 5,
\end{cases}
\end{align*}
while the lowest-order result of Gu \cite{YuGu} basically reads
\begin{align*}
\Big|\Big|u-u_{\shom}-\sum_i \phi_i \partial_i u_{\shom}\Big|\Big|_{L^2} \leq
\begin{cases}
o(\varepsilon) &\text{for }d=3,
\\
o(\varepsilon) &\text{for }d=4,
\\
O(\varepsilon^2) &\text{for }d\geq 5,
\end{cases}
\end{align*}
(where $\phi_i$ denotes the first-order homogenization corrector), a result that improves the bound inferred from the error estimate for the two-scale expansion \cite{GNO3} by a Sobolev embedding
\begin{align*}
\Big|\Big|u-u_{\shom}-\sum_i \phi_i \partial_i u_{\shom}\Big|\Big|_{L^2} \leq
\begin{cases}
O(\varepsilon |\log \varepsilon|) &\text{for }d=2,
\\
O(\varepsilon) &\text{for }d\geq 3.
\end{cases}
\end{align*}
Note that in dimensions $d\geq 5$, Gu constructs homogenization correctors of even higher order and corresponding approximations of $u$ to higher order in terms of the higher-order correctors and solutions to appropriate macroscopic equations. However, also note that the random constants in the estimates of Gu are only shown to have algebraic stochastic moments and that the work of Gu is restricted to symmetric coefficient fields.

Before giving an outline of our strategy, let us give a brief account of quantitative results in stochastic homogenization. The quantitative theory of stochastic homogenization has been initiated by Yurinski\u{\i} \cite{yurinski86}. Naddaf and Spencer were the first to introduce spectral gap inequalities to quantify ergodicity in stochastic homogenization \cite{naddafspencer98}. Gloria and the fourth author \cite{GO1} derived the first optimal estimates on the size of the homogenization error in the linear elliptic case, though with non-optimal stochastic integrability (i.\,e.\ non-optimal stochastic integrability of the constant $\mathcal{C}(a)$ in \eqref{ErrorEstimateuhom}). Armstrong and Smart \cite{armstrongsmart2014} obtained error estimates with optimal stochastic integrability, however with non-optimal estimates on the size of the error (i.\,e.\ non-optimal scaling of the error with respect to $\varepsilon$ in \eqref{ErrorEstimateuhom}). Recently, the picture of error estimates in stochastic homogenization has been completed by Gloria and the fourth author \cite{GloriaOttoNearOptimal} and Armstrong, Kuusi, and Mourrat \cite{ArmstrongKuusiMourrat2016}, who independently derived optimal error estimates with close-to-optimal stochastic integrability.

A more probabilistic viewpoint of stochastic homogenization of linear elliptic equations may be found e.\,g.\ in \cite{MourratVarianceDecay}. For the case of fully nonlinear elliptic equations, we refer to the works of Caffarelli and Souganidis \cite{CaffarelliSouganidis} and Armstrong and Smart \cite{ArmstrongSmartNondivergence}.

A significant simplification in our estimates below relies on the observation that for random elliptic operators, a large-scale Calder\'on-Zygmund-type $L^p$ theory may be developed, see \eqref{I1}. By ergodicity, this $L^p$ theory cheaply provides a moment bound for the gradient of the second-order corrector (see \eqref{proppsimoment}).
Calder\'on-Zygmund-type estimates for random elliptic operators have first been derived by Armstrong and Daniel \cite{ArmstrongDaniel}; in the setting of periodic homogenization, corresponding results have been obtained by Avellaneda and Lin \cite{AvellanedaLinCPAM}. Our estimate \eqref{I1} is a slightly upgraded version of a result by Duerinckx, Gloria, and the fourth author \cite{DuerinckxGloriaOtto}.
Note that further regularity properties of random elliptic operators have been explored in numerous previous works: Marahrens and the fourth author \cite{MO} have established a $C^{0,\alpha}$ regularity theory on large scales. Armstrong and Smart \cite{ArmstrongSmartNondivergence} developed a large-scale Lipschitz regularity theory; motivated by their work, Gloria, Neukamm, and the fourth author \cite{GNO4} have established a $C^{1,\alpha}$ theory. In a recent work of the third and the fourth author \cite{FischerOtto}, higher-order homogenization correctors have been introduced for the first time in the setting of stochastic homogenization to develop a large-scale $C^{k,\alpha}$ regularity theory; however, the focus being on the regularity result, the estimates on the higher-order correctors in \cite{FischerOtto} are non-optimal in the case of fast decorrelation. Closely associated with questions of large-scale regularity are Liouville principles, see for example \cite{ArmstrongKuusiMourrat2016,BellaFehrmanOtto,BenjaminiCopinKozmaYadin,FischerOtto,FischerRaithel,GNO4}. 

\section{Basic Concepts in Homogenization of Elliptic PDEs}

A crucial concept to both deterministic and stochastic homogenization is the concept of \emph{homogenization correctors}. It is an immediate observation that solutions $u_{\shom}$ to a constant-coefficient equation of the form $-\nabla \cdot (a_{\shom}\nabla u_{\shom})=f$ (with, say, smooth right-hand side $f$) lack the microscopic oscillations that are typically present in solutions to the equation $-\nabla \cdot (a\nabla u)=f$ when oscillations on a microscopic scale $\varepsilon$ are present in the coefficient field $a$. For this reason, solutions $u_{\shom}$ to the constant-coefficient equation are not an approximate solution to the microscopic equation: The residual $-\nabla \cdot (a\nabla u_{\shom})-f$ fails to be small, even in the $H^{-1}$ norm.

Homogenization correctors are a tool to adapt solutions $u_{\shom}$ of the constant-coefficient equation to the microscopic oscillations of the coefficient field $a$. The first-order homogenization correctors $\phi_i$, $1\leq i\leq d$, are defined by the requirement that $x\mapsto x_i+\phi_i(x)$ should be a solution to the equation $-\nabla \cdot (a\nabla u)=0$, just like $x\mapsto x_i$ is a solution to the equation $-\nabla \cdot (a_{\shom}\nabla u_{\shom})=0$. The defining equation for the \emph{first-order homogenization corrector} $\phi_i$ therefore reads
\begin{align}
\label{EquationPhi}
-\nabla \cdot (a(e_i+\nabla \phi_i))=0
\quad\quad\text{in }\mathbb{R}^d.
\end{align}
A crucial property of the homogenization corrector $\phi_i$ is its smallness, as compared to the function $x\mapsto x_i$: For example, in the setting of periodic homogenization, the corrector $\phi_i$ may be chosen to satisfy a bound of the form $||\phi_i||_{L^2(\{|x|\leq 1\})}\lesssim \varepsilon$; in the case of stochastic homogenization with sufficiently fast decay of correlations, a similar bound may be achieved for $d>2$ (and for $d=2$, a similar bound may be achieved up to a logarithmic correction).
Throughout the paper, we shall assume $\mathbb{E}[\phi_i]=0$; note that this may be achieved by adding a constant to $\phi_i$ without violating \eqref{EquationPhi}.
For a brief discussion of existence and uniqueness issues associated with the equation \eqref{EquationPhi}, see the end of the present section.

As we shall see below, for general sufficiently smooth functions $u_{\shom}$ the (first-order) \emph{two-scale expansion}
\begin{align}
\label{TwoScaleExpansion}
u_{\shom}+\sum_{i=1}^d \phi_i \partial_i u_{\shom}
\end{align}
provides a modification of $u_{\shom}$ which approximately (in the $H^{-1}$ sense) solves the PDE $-\nabla \cdot (a\nabla u)= -\nabla \cdot (a_{\shom}\nabla u_{\shom})$. In other words, the homogenization correctors provide a means of adapting a sufficiently smooth function $u_{\shom}$ to the microscopic oscillations of the coefficient field $a$ by adding a small correction. Note that both in the periodic setting and in the case of stochastic homogenization with fast decay of correlations and $d\geq 2$, the correction is approximately of the order of $\varepsilon$.

Here, the \emph{effective coefficient} $a_{\shom}$ is characterized by matching the average of the flux in the problem with oscillating coefficients to the flux in the homogenized picture: The vector $a_{\shom} \nabla x_i$ is determined as the average of $a\nabla (x_i+\phi_i)$. In stochastic homogenization, the assumptions of stationarity and ergodicity (see below) entail that averaging over larger and larger scales is equivalent to taking the expectation. Thus, in stochastic homogenization the defining condition of the effective coefficient reads
\begin{align}\nonumber
a_{\shom} e_i := \mathbb{E}[a(e_i+\nabla \phi_i)],
\end{align}
while in periodic homogenization the expectation is replaced by averaging over a single periodicity cell. Let us mention that the constant effective coefficient $a_{\shom}$ inherits the ellipticity and boundedness properties of the coefficient field $a$: If $a$ satisfies $|av|\leq |v|$ and $av\cdot v\geq \lambda |v|^2$ for all vectors $v\in \mathbb{R}^d$, then $a_{\shom}$ satisfies $|a_{\shom}v|\leq d |v|$ and $av\cdot v \geq \lambda |v|^2$ for all $v\in \mathbb{R}^d$.

To estimate the error of the two-scale expansion, it is convenient to introduce a further quantity, a \emph{vector potential} $\sigma_i$ (a skew-symmetric tensor field) for the flux correction $a(e_i+\nabla \phi_i)-a_{\shom} e_i$ in the sense
\begin{align}
\label{EquationSigma}
\nabla \cdot \sigma_{ij} = e_j \cdot \big(a(e_i+\nabla \phi_i)-a_{\shom} e_i\big).
\end{align}
Note that the ``vector potential'' $\sigma_i$ is not uniquely determined by this equation. However, after appropriately fixing the gauge for the vector potential $\sigma_i$, it may again be constructed to be small compared to linearly growing functions like $x\mapsto x_i$: In the case of periodic homogenization, a bound of the form $||\sigma_i||_{L^2(\{|x|\leq 1\}}\lesssim \varepsilon$ holds, and a similar estimate may be achieved in the case of stochastic homogenization with fast decorrelation and $d>2$ (and again for $d=2$ with a logarithmic correction).

Equipped with these definitions, a straightforward computation shows (see below) that the two-scale expansion \eqref{TwoScaleExpansion} gives rise to an approximate solution of the equation $-\nabla \cdot (a\nabla u)=-\nabla \cdot (a_{\shom}\nabla u_{\shom})$ in the sense
\begin{align}
\label{ErrorTwoScale}
&-\nabla \cdot \Big(a \nabla \Big(u_{\shom}+\sum_i \phi_i \partial_i u_{\shom}\Big)\Big)
=-\nabla \cdot (a_{\shom}\nabla u_{\shom})-\nabla \cdot \Big(\sum_i (\phi_i a-\sigma_i) \nabla \partial_i u_{\shom}\Big).
\end{align}
By smallness of $\phi_i$ and $\sigma_i$, the last term in this formula is indeed seen to be small in the $H^{-1}$ sense, at least for sufficiently smooth $u_{\shom}$ with sufficiently fast decay at infinity.

With the previous formula, it is now rather easy to deduce an error estimate for the difference between the solution $u$ to the equation with oscillating coefficients $-\nabla \cdot (a\nabla u)=f$ and the two-scale expansion $u_{\shom}+\sum_i \phi_i \partial_i u_{\shom}$ with the solution $u_{\shom}$ of the (homogenized) constant-coefficient equation $-\nabla \cdot (a_{\shom}\nabla u_{\shom})=f$: The difference $w:=u-(u_{\shom}+\sum_i \phi_i \partial_i u_{\shom})$ satisfies the equation
\begin{align*}
-\nabla \cdot (a\nabla w)=\nabla \cdot \Big(\sum_i (\phi_i a-\sigma_i) \nabla \partial_i u_{\shom}\Big).
\end{align*}
Therefore, a standard energy estimate for this equation entails
\begin{align*}
||\nabla w||_{L^2} \lesssim \sum_i ||(\phi_i a-\sigma_i) \nabla \partial_i u_{\shom}||_{L^2}.
\end{align*}
Both in the setting of periodic homogenization and in the setting of stochastic homogenization with fast decorrelation and $d\geq 3$, the above estimate provides an estimate for the homogenization error $u-u_{\shom}$ of the form
\begin{align*}
||u-u_{\shom}||_{L^2(\{|x|\leq 1\})} \lesssim \varepsilon.
\end{align*}
As already mentioned above, it turns out that the homogenization error -- when measured in the $L^2$ norm -- is actually in general of the order of $\varepsilon$, i.\,e.\ the preceding estimate is sharp: In particular, the oscillations introduced by the term $\sum_i \phi_i \partial_i u_{\shom}$ in the two-scale expansion are actually present in the solution $u$ to the problem with oscillating coefficients. This may for example be seen by the second statement in Theorem~\ref{ErrorEstimateSymmetric} below.

However, by passing to (spatially) weak norms like the $H^{-1}$ norm, it is in fact possible to derive higher-order approximations for the solution $u$ to the problem with oscillating coefficients in terms of solutions to appropriate macroscopic equations (with constant coefficients): In particular, in the periodic setting the $H^{-1}$ norm of the term $\sum_i \phi_i \partial_i u_{\shom}$ in the two-scale expansion is of the order $\varepsilon^2$; for random coefficient fields with fast decorrelation and $d=3$, it is of the order $\varepsilon^{3/2}$. In the rigorous derivation of higher-order approximation results for solutions to the equation with oscillating coefficients, the concept of higher-order homogenization correctors enters. While the first-order homogenization correctors $\phi_i$ are capable of perturbing an affine function $u_{\shom}(x):=\xi\cdot x$ to create an exact solution $u(x):= \xi \cdot x+\sum_i \xi_i \phi_i(x)$ of the equation $-\nabla \cdot (a\nabla u)=-\nabla \cdot (a_{\shom} \nabla u_{\shom})$, they fail at exactly correcting a second-order polynomial $u_{\shom}(x)=Ax\cdot x$ -- in this case, the error term is given by the last term in \eqref{ErrorTwoScale}.

The idea of \emph{second-order correctors} is to add another perturbation $\sum_{ij} \psi_{ij}\partial_i\partial_j u_{\shom}$ to the two-scale expansion in order to be able to exactly correct second-order polynomials. The defining equation of $\psi_{ij}$ therefore reads
\begin{align}
\label{Equationpsi}
-\nabla \cdot (a\nabla \psi_{ij})
=\nabla \cdot \Big((\phi_i a-\sigma_i)e_j\Big).
\end{align}
Note that in the periodic setting, one may construct a second-order corrector subject to the estimate $||\psi_{ij}||_{L^2(\{|x|\leq 1\})} \lesssim \varepsilon^2$. As we shall see below, in the case of stochastic homogenization with fast decorrelation, one can derive a bound on $||\psi_{ij}-\fint_{\{|x|\leq 1\}}\psi_{ij}||_{L^2(\{|x|\leq 1\})}$ that takes the form of the right-hand side in \eqref{ImprovedEstimate}.

Considering now a general sufficiently smooth function $u_{\shom}$, from \eqref{ErrorTwoScale} and \eqref{Equationpsi} we deduce that the second-order two-scale expansion satisfies
\begin{align}
\label{TwoScaleSecondFirstExpression}
&-\nabla \cdot \Big(a\nabla \Big(u_{\shom}+\sum_i \phi_i \partial_i u_{\shom} +\sum_{i,j} \psi_{ij} \partial_i\partial_j u_{\shom} \Big)\Big)
\\&\nonumber
=-\nabla \cdot (a_{\shom}\nabla u_{\shom})
\\&\ \ \ \,\nonumber
-\nabla \cdot \Big(\sum_{i,j} a\psi_{ij} \nabla \partial_i \partial_j u_{\shom}\Big)
-\sum_{i,j} a\nabla \psi_{ij} \cdot \nabla \partial_i \partial_j u_{\shom}
-\sum_i (\phi_i a-\sigma_i) : \nabla^2 \partial_i u_{\shom}.
\end{align}
While from the previous remark on smallness of $\psi_{ij}$ we see that the first error term -- that is the first term in the last line -- is of the desired order, the second and the third error term are not directly seen to be of higher order. Therefore, one would again like to rewrite these terms by introducing a vector potential $\Psi_{ijkl}$ (a tensor field that is skew-symmetric in its last two indices) for the difference between $a\nabla \psi_{ij} \cdot e_k+(\phi_i a-\sigma_i)e_j \cdot e_k$, symmetrized with respect to permutations of $ijk$ (only such a symmetrized version is needed due to the contraction with the third derivative of $u_{hom}$), and its average. It turns out that in the case of symmetric coefficient fields $a$, this average is actually zero: Due to ergodicity, one may treat expectations like averages over large balls, thereby justifying ``integration by parts in expectations'' in the case of test functions with sublinear growth (see for example the argument following \eqref{IntegrationByPartsExpectation}). Thus, one has by the skew-symmetry of $\sigma_{ijk}$ in the last two indices
\begin{align*}
&\operatorname{sym}_{ijk}\mathbb{E}\big[a\nabla \psi_{ij} \cdot e_k+(\phi_i a-\sigma_i)e_j \cdot e_k\big]
\\&
\overset{\eqref{EquationPhi}}{=}
\operatorname{sym}_{ijk}\mathbb{E}\big[-a\nabla \psi_{ij} \cdot \nabla \phi_k+(\phi_i a-\sigma_i)e_j \cdot e_k\big]
\\&
\overset{\eqref{Equationpsi}}{=}
\operatorname{sym}_{ijk}\mathbb{E}\big[(\phi_i a-\sigma_i)e_j \cdot \nabla \phi_k+(\phi_i a-\sigma_i)e_j \cdot e_k\big]
\\&
=
\operatorname{sym}_{ijk}\mathbb{E}\big[\sigma_{ij} \cdot \nabla \phi_k-\sigma_i e_j \cdot e_k+\phi_i a e_j \cdot (e_k+\nabla \phi_k)\big]
\\&
\overset{\eqref{EquationSigma}}{=}
\operatorname{sym}_{ijk}\mathbb{E}\big[-\phi_k e_j\cdot a(e_i+\nabla \phi_i)+a_{\shom,ij} \phi_k-\sigma_i e_j \cdot e_k+\phi_i a e_j \cdot (e_k+\nabla \phi_k)\big]
\\&
=0,
\end{align*}
where in the last step we have assumed that $\phi_i$ has vanishing expectation (note that one may add a constant to $\phi_i$ to enforce this).
Therefore, one introduces a vector potential $\Psi_{ijkl}$, skew-symmetric in the last two indices $k$ and $l$, which satisfies
\begin{align}
\label{EquationPsi}
&\nabla \cdot \Psi_{ijk}
=\operatorname{sym}_{ijk}
\big(a\nabla \psi_{ij} \cdot e_k+(\phi_i a-\sigma_i)e_j \cdot e_k - \varepsilon a_{\shom,ijk}^1\big),
\end{align}
where $a_{\shom,ijk}^1$ is given by
\begin{align}
\label{Defa1ijk}
a_{\shom,ijk}^1:=\frac{1}{\varepsilon}\operatorname{sym}_{ijk}\mathbb{E}[a\nabla \psi_{ij}\cdot e_k+(\phi_i a - \sigma_i) e_j \cdot e_k],
\end{align}
an expression that vanishes for symmetric coefficient fields (here we have introduced the scaling with respect to $\varepsilon$ to ensure that $a_{\shom,ijk}^1$ does not depend on $\varepsilon$). In stochastic homogenization, this vector potential $\Psi$ has been used first in \cite{BellaGiuntiOtto}; note that in \cite{BellaGiuntiOtto} also the existence of solutions $\psi$ and $\Psi$ to the defining equations \eqref{Equationpsi}, \eqref{EquationPsi} with stationary gradients is proven.

In homogenization, the second-order vector potential $\Psi_{ijkl}$ typically admits the same estimates as the second-order corrector $\psi_{ij}$, at least upon fixing the gauge appropriately, which we do by requiring
\begin{align}
\label{Gauge}
-\Delta \Psi_{ijkl} &:= \nabla \cdot ( q^1_{ijl} e_k - q^1_{ijk}e_l )
\end{align}
with
\begin{align*}
q^1_{ijk} &:= \operatorname{sym}_{ijk}[a\nabla \psi_{ij} \cdot e_k + (a\phi_i - \sigma_i) e_j \cdot e_k - \varepsilon a_{\shom,ijk}^1].
\end{align*}
Note that the right-hand side of \eqref{EquationPsi} has vanishing average; in the setting of non-symmetric coefficient fields $a$, this is enforced by the term $-\varepsilon a_{\shom,ijk}^1$. If the average were nonvanishing, one could not hope for a vector potential $\Psi$ with the desired smallness properties, as $\Psi$ would in general contain a linearly growing contribution.

In the case of symmetric coefficient fields $a$, the second-order two-scale expansion now adapts general sufficiently smooth functions $u_{\shom}$ to the oscillating coefficients in the sense that
\begin{align}
\label{ErrorTwoScaleSecond}
&-\nabla \cdot \Big(a\nabla \Big(u_{\shom}+\sum_i \phi_i \partial_i u_{\shom} +\sum_{i,j} \psi_{ij} \partial_i\partial_j u_{\shom} \Big)\Big)
\\&
\nonumber
=-\nabla \cdot (a_{\shom}\nabla u_{\shom})
-\nabla \cdot \Big(\sum_{i,j} (\psi_{ij}a-\Psi_{ij}) \nabla \partial_i \partial_j u_{\shom}\Big).
\end{align}
For nonsymmetric coefficient fields $a$, one additionally needs to compensate for the nonvanishing average of the term $a\nabla \psi_{ij}\cdot e_k+(\phi_i a - \sigma_i) e_j \cdot e_k$, symmetrized in $ijk$, on the right-hand side of the formula \eqref{TwoScaleSecondFirstExpression} (respectively the term $\varepsilon a_{\shom,ijk}^1$ in \eqref{EquationPsi}): One basically additionally needs to add a solution to the equation
\begin{align*}
-\nabla \cdot (a\nabla u^1) =  \underbrace{\sum_{i,j,k}\mathbb{E}[a\nabla \psi_{ij}\cdot e_k+(\phi_i a - \sigma_i) e_j \cdot e_k] \partial_i \partial_j \partial_k u_{\shom}}_{=\sum_{ijk}\varepsilon a_{\shom,ijk}^1 \partial_i \partial_j \partial_k u_{\shom}}
\end{align*}
on the left-hand side to arrive at a right-hand side like in \eqref{ErrorTwoScaleSecond}. As the right-hand side of the equation for $u^1$ is a macroscopically varying quantity, we may again invoke the homogenization ansatz: We solve the effective equation \eqref{EffectiveEquationFirst} instead, the coefficients $a_{\shom,ijk}^1$ being given by
\eqref{Defa1ijk},
and add $\varepsilon u_{\shom}^1+\varepsilon \sum_i \phi_i \partial_i u_{\shom}^1$ instead of $u^1$ (here, we have introduced the scaling with respect to $\varepsilon$ to remove the dependence of $a_{\shom,ijk}^1$ on $\varepsilon$). This yields
\begin{align}\nonumber
&-\nabla \cdot \Big(a\nabla \Big((u_{\shom}+\varepsilon u_{\shom}^1)+\sum_i \phi_i (\partial_i u_{\shom}+\varepsilon \partial_i u_{\shom}^1) +\sum_{i,j} \psi_{ij} \partial_i\partial_j u_{\shom} \Big)\Big)
\\&
\nonumber
=-\nabla \cdot (a_{\shom}\nabla u_{\shom})
-\nabla \cdot \Big(\sum_{i} \varepsilon (a\phi_i - \sigma_i) \nabla \partial_i u_{\shom}^1\Big)
-\nabla \cdot \Big(\sum_{i,j} (\psi_{ij}a-\Psi_{ij}) \nabla \partial_i \partial_j u_{\shom}\Big).
\end{align}
For the detailed derivation of the previous formula, see Section~\ref{DerivationFormulas} below.

Let us now give a few remarks on how to establish existence and smallness of correctors $\phi_i$ and $\sigma_i$, as well as the properties of $a_{\shom}$. In the setting of periodic homogenization, the natural ansatz to construct the corrector $\phi_i$ is to construct it in the class of periodic functions with square-integrable gradient and vanishing mean. In fact, the existence (and uniqueness) of the first-order homogenization correctors $\phi_i$ may be deduced using a simple Lax-Milgram argument in this space; from a scaling argument, one easily infers an estimate of the form $||\phi_i||_{L^2(\{|x|\leq 1\}} \lesssim \varepsilon$.

In contrast, in the setting of stochastic homogenization, obtaining appropriate bounds on the homogenization corrector becomes highly nontrivial and constitutes the most difficult step in the derivation of error estimates: A periodic ansatz is no longer possible and the smallness of the corrector is now caused by stochastic cancellations. Nevertheless, in the case of correlation length $\varepsilon$ (more precisely, finite range of dependence $\varepsilon$), it is possible to deduce estimates of the form
\begin{align*}
\left|\left|\phi_i-\fint_{\{|x|\leq 1\}}\phi_i \right|\right|_{L^2(\{|x|\leq 1\}}
\leq
\begin{cases}
\mathcal{C}(a) \varepsilon^{1/2}&\text{for }d=1,
\\
\mathcal{C}(a) \varepsilon \sqrt{|\log \varepsilon|} &\text{for }d=2,
\\
\mathcal{C}(a) \varepsilon &\text{for }d\geq 3,
\end{cases}
\end{align*}
with a random constant $\mathcal{C}(a)$ satisfying a bound of the form
\begin{align*}
\mathbb{E}[\exp(\mathcal{C}(a)^{2-\delta}/C(\delta))]\leq 2
\end{align*}
for any $\delta>0$ for $d=2$ and $\delta=0$ for $d\geq 3$; see \cite{ArmstrongKuusiMourrat2016,GloriaOttoNearOptimal}.

Note that the equation \eqref{EquationPhi} determines the corrector $\phi_i$ only up to $a$-harmonic functions on $\mathbb{R}^d$. However, $\phi_i$ is determined uniquely up to a random constant by requiring $\nabla \phi_i$ to be a stationary random field with finite second moments $\mathbb{E}[|\nabla \phi_i|^2]<\infty$, and it is this $\phi_i$ that is subject to the aforementioned estimate (and it is this $\phi_i$ that we shall work with throughout the paper). The classical argument to construct such a $\phi_i$ by an argument in probability space may e.\,g.\ be found in \cite{GNO4}, where also $\sigma_{ijk}$ with such properties is constructed. In our setting of fast decorrelation $d(1-\beta)/2>1$ in Definition~\ref{Definition1}, the first-order corrector $\phi_i$ itself exists as a stationary random field. After fixing the expectation $\mathbb{E}[\phi_i]=0$, the corrector $\phi_i$ is then uniquely determined by \eqref{EquationPhi} and the condition of stationarity.
A similar existence and uniqueness issue arises for the second-order corrector $\psi_{ij}$ (and also the vector potential $\Psi_{ijkl}$), see \cite{BellaGiuntiOtto} or \cite{YuGu} for the analogous construction of $\psi_{ij}$ for which $\nabla \psi_{ij}$ is a stationary random field with finite second moment (again, it is this $\psi_{ij}$ that we shall work with throughout the paper).

\section{Main Results}

In order to provide a quantification of the decay of correlations in the coefficient field $a$ (that is, a quantification of the ergodicity of the ensemble), we employ a coarsened logarithmic Sobolev inequality introduced by Gloria, Neukamm, and the fourth author \cite{GNO4}.
\begin{definition}
\label{Definition1}
Let $\langle\cdot\rangle$ be an ensemble, that is a probability measure on the space of coefficient fields $a:\mathbb{R}^d\rightarrow \mathbb{R}^{d\times d}$.

We call the ensemble bounded and uniformly elliptic if it is concentrated on the $\lambda$-uniformly elliptic coefficient fields for some $\lambda\in (0,1]$, i.\,e.\ if with probability one we have $a(x)v \cdot v\geq \lambda |v|^2$ and $|a(x)v|\leq |v|$ for all $v\in \mathbb{R}^d$ and all $x\in \mathbb{R}^d$.

We say that the ensemble is stationary if it is invariant under spatial shifts, i.\,e.\ if $a(\cdot+x)$ has the same law as $a(\cdot)$ for any $x\in \mathbb{R}^d$.

We say that the ensemble satisfies a coarsened logarithmic Sobolev inequality (LSI) with exponent $\beta\in [0,1)$ and correlation length $\varepsilon$ if for any random variable $\xi=\xi(a)$ the estimate
\begin{align}
\label{LSI}
\mathbb{E}[\xi^2 \log \xi^2]-\mathbb{E}[\xi^2]\mathbb{E}[\log \xi^2]
\leq \mathbb{E} \bigg[\left|\left|\frac{\partial\xi}{\partial a}\right|\right|_{CC}^2\bigg]
\end{align}
is satisfied, where the carr\'e-du-champ is defined by
\begin{align*}
\left|\left|\frac{\partial\xi}{\partial a}\right|\right|_{CC}^2
:=
\sum_D \left(\int_{D} \left|\frac{\partial \xi}{\partial a}(x)\right| \,dx \right)^2
\end{align*}
for a partition $(D_i)$ of $\mathbb{R}^d$ satisfying
\begin{align*}
\left(\frac{\operatorname{dist}(D_i,0)}{\varepsilon}+1\right)^\beta \leq \frac{\operatorname{diam}(D_i)}{\varepsilon}\leq C(d) \left(\frac{\operatorname{dist}(D_i,0)}{\varepsilon}+1\right)^\beta.
\end{align*}
Here, $\frac{\partial \xi}{\partial a}$ is the Fr\'echet derivative of $\xi$ with respect to the coefficient field $a=a(x)$.
\end{definition}
Note that the case $\beta=0$ (and for example a corresponding partition of $\mathbb{R}^d$ into cubes of equal size of the form $\{D_i:i\in \mathbb{N}\}:=\{[0,\varepsilon]^d+\varepsilon z:z\in \mathbb{Z}^d\}$) corresponds to a non-coarsened (classical) logarithmic Sobolev inequality. An example of an ensemble satisfying the logarithmic Sobolev inequality with $\beta=0$ and correlation length $C(d)\varepsilon$ would be the ensemble of coefficient fields that
\begin{itemize}
\item is given by $a(x):=\xi(\tilde a(x))$, where
\item $\xi:\mathbb{R} \rightarrow \mathbb{R}^{d\times d}$ is a bounded Lipschitz map taking values in the set of $\lambda$-uniformly elliptic matrices for some $\lambda>0$, and where
\item $\tilde a(x)$ is a stationary, centered, Gaussian random field subject to the covariance estimate
\begin{align*}
\int_{\mathbb{R}^d} \sup_{|y|=|x|} \Big| \mathbb{E}\big[\tilde a(y) \tilde a(0)\big] \Big| \,dx \leq \varepsilon^d.
\end{align*}
\end{itemize}
For a depiction of one realization of such a coefficient field, see Figure~\ref{GaussianCoefficientField}. A proof of the logarithmic Sobolev inequality for such an ensemble is provided in \cite[Lemma 4]{GNO4}.

\begin{figure}
~
\includegraphics[scale=0.3]{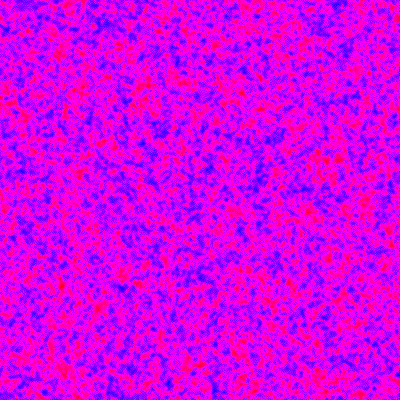}
\caption{A random coefficient field generated with the help of a Gaussian random field.\label{GaussianCoefficientField}}
\end{figure}

For symmetric coefficient fields, our main result is the following estimate on the homogenization error, measured in the $H^{-1}$ norm.
\begin{theorem}
\label{ErrorEstimateSymmetric}
Let $\langle\cdot \rangle$ be a bounded, uniformly elliptic, and stationary ensemble of coefficient fields on $\mathbb{R}^d$, $d\geq 3$; assume that $a$ is almost surely symmetric, i.\,e.\ satisfies $a^T(x)=a(x)$ for all $x\in \mathbb{R}^d$ almost surely. Suppose that $\langle\cdot \rangle$ satisfies a coarsened logarithmic Sobolev inequality with exponent $\beta\in [0,1-2/d)$ and correlation length $\varepsilon\leq \frac{1}{2}$ in the sense of Definition~\ref{Definition1}. Then there exists a nonnegative random constant $\mathcal{C}(a)$ satisfying the bound
\begin{align*}
\mathbb{E}[\exp(\mathcal{C}(a)^{1/C(d,\lambda,\beta)})] \leq 2
\end{align*}
for some constant $C(d,\lambda,\beta)$ such that for all $f\in L^2(\mathbb{R}^d)$ which are supported in $\{|x|\leq 1\}$ the following homogenization error estimate holds true:

The solution $u_{\shom}$ to the equation
\begin{align*}
-\nabla \cdot (a_{\shom}\nabla u_{\shom})=f
\end{align*}
approximates the solution $u$ to the equation
\begin{align*}
-\nabla \cdot (a\nabla u) =f
\end{align*}
in the sense
\begin{align*}
||u-u_{\shom}||_{H^{-1}(\{|x|\leq 1\})}
\leq
\begin{cases}
\mathcal{C}(a) ||\nabla^3 u_{\shom}||_{L^\infty} \varepsilon^{d(1-\beta)/2} &\text{for }d(1-\beta)/2<2,
\\
\mathcal{C}(a) ||\nabla^3 u_{\shom}||_{L^\infty} \varepsilon^2 |\log \varepsilon| &\text{for }d(1-\beta)/2=2,
\\
\mathcal{C}(a) ||\nabla^3 u_{\shom}||_{L^\infty} \varepsilon^2 &\text{for }d(1-\beta)/2>2.
\end{cases}
\end{align*}
For the second-order two-scale expansion, the error bound at the level of the gradient
\begin{align*}
\bigg|\bigg|\nabla u-\nabla &\Big(u_{\shom}+\sum_i \phi_i \partial_i u_{\shom} + \sum_{i,j} \psi_{ij} \partial_i \partial_j u_{\shom} \Big)\bigg|\bigg|_{L^2(\mathbb{R}^d)}
\\&
\leq
\begin{cases}
\mathcal{C}(a) ||\nabla^3 u_{\shom}||_{L^\infty} \varepsilon^{d(1-\beta)/2} &\text{for }d(1-\beta)/2<2,
\\
\mathcal{C}(a) ||\nabla^3 u_{\shom}||_{L^\infty} \varepsilon^2 \sqrt{|\log \varepsilon|} &\text{for }d(1-\beta)/2=2,
\\
\mathcal{C}(a) ||\nabla^3 u_{\shom}||_{L^\infty} \varepsilon^2 &\text{for }d(1-\beta)/2>2,
\end{cases}
\end{align*}
is satisfied.
\end{theorem}
For non-symmetric coefficient fields, in order to obtain a higher-order approximation of the solution to the problem with oscillating coefficients, one needs to add the solution of another macroscopic problem to the solution of the effective equation. With this additional correction, we obtain an analogous higher-order error bound in the nonsymmetric setting:
\begin{theorem}
\label{ErrorEstimateNonsymmetric}
Let $\langle\cdot \rangle$ be a bounded, uniformly elliptic, and stationary ensemble of coefficient fields on $\mathbb{R}^d$, $d\geq 3$. Suppose that $\langle\cdot \rangle$ satisfies a coarsened logarithmic Sobolev inequality with exponent $\beta\in [0,1-2/d)$ and correlation length $\varepsilon\leq \frac{1}{2}$ in the sense of Definition~\ref{Definition1}.
Define the coefficient $a_{\shom,ijk}^1$ as
\begin{align*}
a_{\shom,ijk}^1:=\frac{1}{\varepsilon} \mathbb{E}[e_k\cdot a\nabla \psi_{ij}+(\phi_i a_{kj}-\sigma_{ikj})].
\end{align*}
Then there exists a constant $C(d,\beta,\lambda)$ such that $|a_{\shom,ijk}^1|\leq C(d,\beta,\lambda)$ holds.
Furthermore, there exists a nonnegative random constant $\mathcal{C}(a)$ satisfying the bound
\begin{align*}
\mathbb{E}[\exp(\mathcal{C}(a)^{1/C(d,\lambda,\beta)})] \leq 2
\end{align*}
for some constant $C(d,\lambda,\beta)$ such that for all $f\in L^2(\mathbb{R}^d)$ which are supported in $\{|x|\leq 1\}$ the following homogenization error estimate holds true:

The solution $u_{\shom}$ to the equation
\begin{align*}
-\nabla \cdot (a_{\shom}\nabla u_{\shom})=f
\end{align*}
together with the solution $u_{\shom}^1$ of the macroscopic equation
\begin{align*}
-\nabla \cdot (a_{\shom} \nabla u_{\shom}^1) = \sum_{i,j,k} a_{\shom,ijk}^1 \partial_i \partial_j \partial_k u_{\shom}
\end{align*}
approximate the solution $u$ to the equation
\begin{align*}
-\nabla \cdot (a\nabla u) =f
\end{align*}
in the sense
\begin{align*}
||u-(u_{\shom}+\varepsilon u_{\shom}^1)&||_{H^{-1}(\{|x|\leq 1\})}
\\&
\leq
\begin{cases}
\mathcal{C}(a) (||\nabla^3 u_{\shom}||_{L^\infty}+||\nabla^2 u_{\shom}^1||_{L^\infty}) \varepsilon^{d(1-\beta)/2} &\text{for }d(1-\beta)/2<2,
\\
\mathcal{C}(a) (||\nabla^3 u_{\shom}||_{L^\infty}+||\nabla^2 u_{\shom}^1||_{L^\infty}) \varepsilon^2 |\log \varepsilon| &\text{for }d(1-\beta)/2=2,
\\
\mathcal{C}(a) (||\nabla^3 u_{\shom}||_{L^\infty}+||\nabla^2 u_{\shom}^1||_{L^\infty}) \varepsilon^2 &\text{for }d(1-\beta)/2>2.
\end{cases}
\end{align*}
For the second-order two-scale expansion, the error bound at the level of the gradient
\begin{align*}
\bigg|\bigg|\nabla u-\nabla \Big(u_{\shom}+&\varepsilon u_{\shom}^1+\sum_i \phi_i \partial_i (u_{\shom}+\varepsilon u_{\shom}^1) + \sum_{i,j} \psi_{ij} \partial_i \partial_j u_{\shom} \Big)\bigg|\bigg|_{L^2(\mathbb{R}^d)}
\\&
\leq
\begin{cases}
\mathcal{C}(a) (||\nabla^3 u_{\shom}||_{L^\infty}+||\nabla^2 u_{\shom}^1||_{L^\infty}) \varepsilon^{d(1-\beta)/2} &\text{for }d(1-\beta)/2<2,
\\
\mathcal{C}(a) (||\nabla^3 u_{\shom}||_{L^\infty}+||\nabla^2 u_{\shom}^1||_{L^\infty}) \varepsilon^2 \sqrt{|\log \varepsilon|} &\text{for }d(1-\beta)/2=2,
\\
\mathcal{C}(a) (||\nabla^3 u_{\shom}||_{L^\infty}+||\nabla^2 u_{\shom}^1||_{L^\infty}) \varepsilon^2 &\text{for }d(1-\beta)/2>2,
\end{cases}
\end{align*}
is satisfied.
\end{theorem}

Our homogenization error estimates are based on the following estimate on the second-order homogenization corrector $\psi$ and $\Psi$. Note that we formulate this theorem on the microscopic scale, i.\,e.\ we set $\varepsilon:=1$, as it is a result for the corrector on the full space $\mathbb{R}^d$ which otherwise would lack a natural scale.
\begin{theorem}
\label{SecondCorrectorEstimate}
Let $\langle\cdot \rangle$ be a bounded, uniformly elliptic, and stationary ensemble of coefficient fields on $\mathbb{R}^d$, $d\geq 3$. Suppose that $\langle\cdot \rangle$ satisfies a coarsened logarithmic Sobolev inequality with correlation length $\varepsilon:=1$ and exponent $\beta\in [0,1-2/d)$ in the sense of Definition~\ref{Definition1}. Then for any $r\geq 1$ there exists a nonnegative random number $\mathcal{C}_r(a)$ with
\begin{align*}
\mathbb{E}[\exp(\mathcal{C}_r(a)^{1/C(d,\lambda,\beta)})] \leq 2
\end{align*}
for some constant $C(d,\lambda,\beta)$ such that the second-order homogenization corrector $\psi$ defined by \eqref{Equationpsi} and the corresponding vector potential for the flux correction $\Psi$ defined by \eqref{EquationPsi} with the choice of gauge \eqref{Gauge} satisfy an estimate of the form
\begin{align*}
\left(\fint_{\{|x|\leq r\}} \left|\psi-\fint_{\{|x|\leq r\}} \psi\right|^2+\left|\Psi-\fint_{\{|x|\leq r\}}\Psi\right|^2 \,dx\right)^{1/2}
\leq
\begin{cases}
\mathcal{C}_r(a) r^{2-d(1-\beta)/2} & \text{for }d(1-\beta)/2 < 2,
\\
\mathcal{C}_r(a) \sqrt{\log (r+1)} & \text{for }d(1-\beta)/2 = 2,
\\
\mathcal{C}_r(a) &\text{for }d(1-\beta)/2 > 2.
\end{cases}
\end{align*}
\end{theorem}
We would like to remark that in the case $d(1-\beta)/2>2$, the second-order corrector $\psi$ and the corresponding vector potential $\Psi$ exist as stationary random fields; for a proof of this fact for $\psi$ in the case $\beta=0$, see e.\,g.\ \cite{YuGu}.

Note that the bound on the second-order provided by Theorem~\ref{SecondCorrectorEstimate} is precisely the input required by \cite{BellaGiuntiOtto} to obtain a localized estimate on the homogenization error -- that is, an estimate down to a (random) scale $r_\ast$ (the typical value of $r_\ast$ being of the order of $\varepsilon=1$). Note that for the result in \cite{BellaGiuntiOtto} one needs an estimate on the second-order corrector with a random constant that is uniform in $r$, an estimate that one may obtain from Theorem~\ref{SecondCorrectorEstimate} by giving up a power of $\log \log r$.

We would like to emphasize that just like in \cite{GNO4}, all of our results are also valid in the case of elliptic systems. In the systems' case, one just needs to add the required additional indices to the coefficients, correctors, and solutions both in the statement of the results and in the proofs.

\section{Strategy of the Proof}

As mentioned in the introduction, the key ingredient of the present work is the estimate on the second-order corrector established in Theorem \ref{SecondCorrectorEstimate}. We plan to partly mimic the approach of Gloria, Neukamm, and the fourth author concerning the first-order corrector $(\phi,\sigma)$ in \cite{GNO4}, first deriving sensitivity estimates on ``averages'' of $(\nabla \psi, \nabla \Psi)$ and then using them in the logarithmic Sobolev inequality. However, in contrast to the approach for the first-order corrector in \cite{GNO4}, in the case of the second-order corrector an interesting simplification is possible due to the fact that we already have access to large-scale regularity results for the random operator.

The logarithmic Sobolev inequality (abbreviated: LSI) \eqref{LSI}, by which we quantify the ergodicity of our ensemble of coefficient fields $a$, converts sensitivity estimates for a zero-mean random variable (that is, estimates on the Fr\'echet derivative of a zero-mean random variable $\xi=\xi(a)$) into estimates on the stochastic moments of the random variable. Our intention is to exploit the LSI to convert sensitivity estimates for the second-order corrector $\psi$ and the corresponding vector potential $\Psi$ into estimates on $\psi$ and $\Psi$ themselves. To this aim, we shall require an $L^p$ version of the LSI, which (as shown in \cite{GNO4}) is actually a consequence of the LSI in its original form \eqref{LSI}.
\begin{lemma}[Lemma 6 in \cite{GNO4}]
\label{LpLSI}
The assumption \eqref{LSI} entails for all $q\geq 1$ the estimate
\begin{align*}
\mathbb{E}\big[\big|\xi-\mathbb{E}[\xi]\big|^{2q}\big]^{1/q}
\leq Cq \mathbb{E}\bigg[\bigg|\bigg|\frac{\partial \xi}{\partial a}\bigg|\bigg|_{CC}^{2q}\bigg]^{1/q}
\end{align*}
for a generic (universal) constant $C$.
\end{lemma}

As the LSI only provides estimates on stochastic moments for random variables with zero mean, we shall not estimate norms of the second-order corrector $(\psi,\Psi)$ directly, but rather start by estimating appropriate integral averages of the gradient of the second-order corrector. To this aim, we start with a deterministic estimate on the sensitivity of the gradient of the second-order corrector $\nabla \psi$ and the gradient of the corresponding vector potential $\nabla \Psi$ with respect to changes in the coefficient field $a$. This estimate is an analogue of the corresponding estimate \cite[Proposition 3]{GNO4} for the first-order corrector $\phi$ and the vector potential $\sigma$.

\begin{proposition}\label{prop1}
Let $\langle \cdot \rangle$ be a stationary, bounded, and uniformly elliptic ensemble of coefficient fields on $\mathbb{R}^d$, $d\geq 3$. Suppose that $\langle\cdot\rangle$ satisfies a coarsened logarithmic Sobolev inequality with exponent $\beta\in [0,1-\frac{2}{d})$ and correlation length $1$.

Consider the minimal radius above which a mean-value property (or, more precisely, a $C^{0,1}$ regularity theory) holds for $a$-harmonic functions and $a^*$-harmonic functions, that is, the random variable $r_*\ge 1$ characterized by 
\begin{equation}\label{g.6}
r_*:=\max_{a,a^*} \inf\left\{\;r=2^m \ge 1 \;\Big|\;\forall\;\rho=2^M\ge r:\quad
\frac{1}{r^2}\fint_{\{|x|\leq r\}}\left|(\phi,\sigma)-\fint_{\{|x|\leq r\}} (\phi,\sigma)\right|^2\le\frac{1}{C(d,\lambda)}\right\}
\end{equation}
with the understanding that $r_*=+\infty$ if the set is empty and where $\max_{a,a^*}$ denotes the maximum of the quantity for the coefficient fields $a$ and $a^*$ (with $a^*$ denoting the pointwise transpose of $a$).
%

Consider a linear functional $h\mapsto Fh$ on vector fields $\mathbb{R}^d\ni x\mapsto h(x)$ 
given by
\begin{equation*}
Fh\,=\,\int h\cdot g,
\end{equation*}
where  $g$ is supported in $\{|x|\leq r\}$ for some radius $r\ge 1$.
\smallskip

Then for a Meyers exponent $\bar p>1$ (depending only on $\lambda$ and $d$), any $p$ with $1<p< \bar p$, and any $\alpha$ with $2p \le \alpha < d(2p-1)$, we have for any $g$ with
\begin{equation*}
  \Big(\fint_{\{|x|\leq r\}} |g|^{2p}\Big)^\frac{1}{2p}\,\leq \, r^{-d}
\end{equation*}
the estimate
\begin{align}
\label{s.S2b}
\Big|\Big|\frac{\partial F\nabla(\psi,\Psi)}{\partial a}\Big|\Big|_{CC}^2
\lesssim &  \left(\int|g|^{2p}\right)^{\frac{1}{p}}
\times \biggl(\sum_D (r_* + r)^{\frac{2p}{p-1}} |D| (\min_D\omega_{\alpha-2p})^{-\frac{1}{p-1}} \biggl( \int_D|\nabla\phi+e|^2 \biggr)^{\frac{p}{p-1}}\biggr)^{\frac{p-1}{p}}
\\&
\nonumber
+\left(\int|g|^{2p}\right)^{\frac{1}{p}}\times
      \biggl(\sum_D|D| (\min_D\omega_{\alpha})^{-\frac{1}{p-1}} \biggl( \int_D|\nabla\psi|^2 + |\phi|^2 \biggr)^{\frac{p}{p-1}}\biggr)^{\frac{p-1}{p}},
\end{align}
where by $F\nabla(\psi,\Psi)$ we denote any functional of the form $\int \nabla \psi_{ij}\cdot g$ or $\int\nabla \Psi_{ijkl}\cdot g$ and where the weight functions are defined by
\begin{equation}
\label{DefWeights}
 \omega_\gamma(x) := \left(\frac{|x|}{r_*+r}+1\right)^\gamma. 
\end{equation}
Here, by $\lesssim$ we mean up to a constant factor depending only on $d$, $\lambda$, $p$, and $\alpha$.
%
%
\end{proposition}

The preceding estimate on the sensitivity of the second-order corrector is based on the following weighted Meyers estimate for the operator $-\nabla \cdot a^*\nabla$ (that is, the formal adjoint operator to $-\nabla \cdot a \nabla$; here, $a^*$ denotes the pointwise transpose of $a$).

\begin{lemma}
\label{WeightedMeyers}
Let $a$ be a $\lambda$-uniformly elliptic coefficient field on $\mathbb{R}^d$. Let $r>0$ be arbitrary and let $r_*$ be as in \eqref{g.6}.
Let $v_0\in H^1_{loc}(\R^d)$ and $g_0 \in L^2_{loc}(\R^d;\R^d)$ be decaying functions related through
\begin{equation}\label{e1.17a}
-\nabla \cdot (a^* \nabla v) = \nabla \cdot g
\end{equation}
and let $\hat v\in H^1_{loc}(\R^d)$ and $\hat g\in L^2_{loc}(\R^d)$ be decaying functions related through
\begin{equation}\label{e1.17b}
-\nabla \cdot (a^* \nabla \hat v) = \hat g.
\end{equation}
There exists a Meyers exponent $\bar p>1$, which only depends on $d$ and $\lambda$, such that for all $1\leq p<\bar p$, all $\alpha_1 < d(2p-1)$, and all $0 < \alpha_0 < \alpha_1$ we have
\begin{equation}\label{e1.13}
    \left(\int |\nabla v|^{2p}\omega_{\alpha_0} \right)^{\frac{1}{2p}}\lesssim \left(\int |g|^{2p}\omega_{\alpha_1}\right)^{\frac{1}{2p}}
\end{equation}
and for all $1\leq p<\bar p$, all $\alpha_1 < d(2p-1)$, and all $0<\alpha_0<\alpha_1-2p$ we have
\begin{equation}\label{e1.13++}
    \left(\int |\nabla \hat v|^{2p}\omega_{\alpha_0} \right)^{\frac{1}{2p}}
    \lesssim 
    (r_* + r) \left(\int |\hat g|^{2p}\omega_{\alpha_1}\right)^{\frac{1}{2p}},
\end{equation}
where the weights $\omega_{\alpha_i}$ are defined by \eqref{DefWeights} and where
$\lesssim$ means less or equal up to a constant factor that only depends on $d,\lambda, p, \alpha_0$, and $\alpha_1$.
\end{lemma}

The estimate in Proposition~\ref{prop1} entails the following bound on stochastic moments of the sensitivity of the second-order corrector. Note that the right-hand side in particular involves stochastic moments of $\nabla \psi$, $\nabla \phi$, and $\phi$; while the moments of the latter two quantities are controlled due to the known results for the first-order corrector, a bound on stochastic moments of $\nabla \psi$ is yet to be established. While the second-order corrector $\psi$ and the vector potential $\Psi$ are not uniquely defined by \eqref{Equationpsi} and \eqref{Gauge}, uniqueness (up to a random constant) is ensured by imposing the conditions that $\nabla \psi$ and $\nabla \Psi$ be stationary random fields with $\mathbb{E}[\nabla \psi]=0$, $\mathbb{E}[\nabla \Psi]=0$, and $\mathbb{E}[|\nabla \psi|^2+|\nabla \Psi|^2]<\infty$.

\begin{proposition}
\label{StochasticMomentsSensitivity}
Let $\langle \cdot \rangle$ be a stationary, bounded, and uniformly elliptic ensemble of coefficient fields on $\mathbb{R}^d$, $d\geq 3$. Suppose that $\langle\cdot\rangle$ satisfies a coarsened logarithmic Sobolev inequality with exponent $\beta\in [0,1-\frac{2}{d})$ and correlation length $1$.

Let $0<p-1\ll 1$ and $m\gg 1$. Consider a linear functional $h\mapsto Fh$ on vector fields $\mathbb{R}^d\ni x\mapsto h(x)$ given by
\begin{equation*}
Fh\,=\,\int h\cdot g,
\end{equation*}
where $g$ is supported in $\{|x|\leq r\}$ for some radius $r\ge 1$ and satisfies
\begin{equation*}
  \Big(\fint_{\{|x|\leq r\}} |g|^{2p}\Big)^\frac{1}{2p}\,\leq \, r^{-d}.
\end{equation*}
\smallskip
The second-order corrector $\psi$ and the corresponding vector potential $\Psi$ are subject to the sensitivity estimate
\begin{align*}
\mathbb{E}\left[\Big|\Big|\frac{\partial F\nabla(\psi,\Psi)}{\partial a}\Big|\Big|_{CC}^m \right]^{1/m}
\lesssim&
~r^{-(1-\beta)d/2+1}
\mathbb{E}\left[(1+r_\ast)^{m\alpha/p}\right]^{1/2m}
\mathbb{E}\left[\biggl(\fint_{\{|x|\leq 1\}}|\nabla\phi+e|^2 \biggr)^{m} \right]^{1/2m}
\\&
+r^{-(1-\beta)d/2} \mathbb{E}\left[(1+r_\ast)^{m\alpha/p}\right]^{1/2m}
\mathbb{E}\left[\biggl(\fint_{\{|x|\leq 1\}}|\nabla\psi|^2 + |\phi|^2 \biggr)^{m}\right]^{1/2m}.
\end{align*}
\end{proposition}

We now provide the estimate on stochastic moments of the gradient of the second-order corrector $\nabla \psi$ that is needed to convert the statement in the previous proposition into an actual bound on stochastic moments of the sensitivity of the functionals.

\begin{proposition}\label{proppsimoment}
Let $\langle \cdot \rangle$ be a stationary, bounded, and uniformly elliptic ensemble of coefficient fields on $\mathbb{R}^d$, $d\geq 3$. Suppose that $\langle\cdot\rangle$ satisfies a coarsened logarithmic Sobolev inequality with exponent $\beta\in [0,1-\frac{2}{d})$ and correlation length $1$.
Then for any $p \in [2,\infty)$ the estimate
 \begin{equation}\nonumber
  \mathbb{E}\Bigg[ \biggl( \fint_{\{|x|\leq 1\}} |\nabla \psi|^2 + |\nabla \Psi|^2 \biggr)^{\frac p2} \Bigg]^{\frac1p} \leq C(d,\beta,\lambda) p^{C(d,\beta,\lambda)}
\end{equation}
holds.
\end{proposition}
This estimate on the stochastic moments of the gradient of the second-order corrector is based on the following large-scale Calderon-Zygmund-type $L^p$ theory for operators with random coefficients.
\begin{proposition}\label{propCZ}
Let $\langle \cdot \rangle$ be a stationary, bounded, and uniformly elliptic ensemble of coefficient fields on $\mathbb{R}^d$, $d\geq 3$. Suppose that $\langle\cdot\rangle$ satisfies a coarsened logarithmic Sobolev inequality with exponent $\beta\in [0,1-\frac{2}{d})$ and correlation length $1$.

There exists a threshold $\delta=\delta(d,\lambda)>0$ and a field 
$0<\underline{r_*}\le r_*$ (with $r_*$ from \eqref{g.6}) such that for any suitably decaying 
scalar field $u$ and vector field $g$ related by
\begin{equation}\nonumber
-\nabla\cdot (a \nabla u)= \nabla \cdot g
\end{equation}
and any exponent $1<p<\infty$ we have
\begin{equation}\label{I1}
\biggl( \int\Big(\fint_{B_{\underline{r_*}}(x)}|\nabla u|^2\Big)^p\dx\biggr)^{\frac 1p} \lesssim p^2  \biggl( \int \biggl( \fint_{B_1(x)} |g|^2 \biggr)^p \dx \biggr)^{\frac 1p}
\end{equation}
where $\lesssim$ means $\le C$ for a constant only depending on $d$ and $\lambda$.
\end{proposition}

Putting Proposition \ref{StochasticMomentsSensitivity} and Proposition \ref{proppsimoment} together, by the $L^p$ version of the logarithmic Sobolev inequality (see Lemma~\ref{LpLSI}) we infer the following estimate on stochastic moments of ``spatial averages'' of the second-order corrector $\psi$ and the corresponding vector potential $\Psi$.
\begin{proposition}
\label{MomentBoundsCorrectorAverages}
Let $\langle \cdot \rangle$ be a stationary, bounded, and uniformly elliptic ensemble of coefficient fields on $\mathbb{R}^d$, $d\geq 3$. Suppose that $\langle\cdot\rangle$ satisfies a coarsened logarithmic Sobolev inequality with exponent $\beta\in [0,1-\frac{2}{d})$ and correlation length $1$.

Let $0<p-1\ll 1$ and let $r\geq 1$. Consider a linear functional $h\mapsto Fh$ on vector fields $\mathbb{R}^d\ni x\mapsto h(x)$ given by
\begin{equation*}
Fh\,=\,\int h\cdot g,
\end{equation*}
where $g$ is supported in $\{|x|\leq r\}$ and satisfies
\begin{equation*}
  \Big(\fint_{\{|x|\leq r\}} |g|^{2p}\Big)^\frac{1}{2p}\,\leq \, r^{-d}.
\end{equation*}
Then there exists a constant $C(d,\lambda,\beta)$ such that the stochastic moments of the functional $F\nabla (\psi,\Psi)$ (by which we understand any functional of the form $F\nabla \psi_{ij}$ or $F\nabla \Psi_{ijkl}$) are estimated for any $m\geq \frac{2p}{p-1}$ by
\begin{align*}
\mathbb{E}
\left[
\big|F\nabla(\psi,\Psi)\big|^m
\right]^{1/m}
\leq&
C(d,\lambda,\beta) m^{C(d,\lambda,\beta)}  r^{-(1-\beta)d/2+1}.
\end{align*}
\end{proposition}

Finally, the next lemma provides a means to translate the estimates on ``averages'' of the gradient of the second-order corrector $\psi$, which Proposition \ref{MomentBoundsCorrectorAverages} provides, into estimates on the corrector itself (the same applying also to the vector potential $\Psi$).

\begin{lemma}
\label{IntegrateLemma}
Let $\gamma>0$, $m\geq 2$, and $K\geq 0$. Let $u=u(x,a)$ be a random function subject to the estimates
\begin{align}
\label{IntLemma.AssumptionGradient}
\mathbb{E}\left[\left(\fint_{\{|x-x_0|\leq 1\}} |\nabla u|^2 \,dx\right)^{m/2}\right]^{1/m}
\leq K
\quad\quad\text{for all }x_0\in \mathbb{R}^d
\end{align}
and
\begin{align}
\label{IntLemma.AssumptionGradientAverages}
\mathbb{E}\left[\left(\int \nabla u \cdot g \,dx\right)^m \right]^{1/m}
\leq \frac{K}{r^\gamma}
\end{align}
for all $r\geq 1$, all $x_0\in \mathbb{R}^d$, and all vector fields $g:\mathbb{R}^d \rightarrow \mathbb{R}^d$ supported in $\{|x-x_0|\leq r\}$ satisfying
\begin{align}\nonumber
\Big(\fint_{\{|x-x_0|\leq r\}} |g|^{2+1/d} \,dx\Big)^{1/(2+1/d)}\,\leq \, r^{-d}.
\end{align}
Then estimates of the form
\begin{align}\nonumber
\mathbb{E}\left[\left(\fint_{\{|x|\leq r\}} \left|u-\fint_{\{|x|\leq r\}}u\right|^2 \,dx\right)^{m/2} \right]^{1/m}
\leq
\begin{cases}
C_3 K r^{1-\gamma} &\text{for }\gamma<1,
\\
C_3 K \sqrt{\log(r+2)} &\text{for }\gamma=1,
\\
C_3 K &\text{for }\gamma>1,
\end{cases}
\end{align}
and
\begin{align}
\label{IntLemma.Result2}
\mathbb{E}\left[\left(r^{-d/2}\left|\left|u-\fint_{\{|x|\leq r\}}u\right|\right|_{H^{-1}(\{|x|\leq r\})} \right)^{m} \right]^{1/m}
\leq
\begin{cases}
C_3 K r^{2-\gamma} &\text{for }\gamma<2,
\\
C_3 K \log(r+2) &\text{for }\gamma=2,
\\
C_3 K &\text{for }\gamma>2,
\end{cases}
\end{align}
hold for all $r\geq 1$, the constant $C_3$ depending on $\gamma$, but being independent of $m$.
\end{lemma}
Combining the previous lemma with the estimates from \cite[Theorem 2]{GNO4}, we infer a bound on the $H^{-1}$ norm of the first-order corrector $\phi$ that is of significantly better order than the bound on the $L^2$ norm.
\begin{lemma}
\label{EstimateCorrectorH-1}
Let $\langle \cdot \rangle$ be a stationary, bounded, and uniformly elliptic ensemble of coefficient fields on $\mathbb{R}^d$, $d\geq 3$. Suppose that $\langle\cdot\rangle$ satisfies a coarsened logarithmic Sobolev inequality with exponent $\beta\in [0,1-\frac{2}{d})$ and correlation length $1$.

Then for every $r\geq 1$ there exists a random constant $\mathcal{C}_r(a)$ with
\begin{align*}
\mathbb{E}[\exp(\mathcal{C}_r(a)^{1/C(d,\lambda,\beta)})] \leq 2
\end{align*}
for some constant $C(d,\lambda,\beta)$ such that the stationary first-order homogenization corrector $\phi$, uniquely defined by \eqref{EquationPhi} and the requirements $\mathbb{E}[\phi_i]=0$ and $\mathbb{E}[|\nabla \phi_i|^2]<\infty$, satisfies an estimate of the form
\begin{align*}
r^{-d/2} \left|\left|\phi\right|\right|_{H^{-1}(\{|x|\leq r\})}
\leq
\begin{cases}
\mathcal{C}_r(a) r^{2-d(1-\beta)/2} & \text{for }d(1-\beta)/2 < 2,
\\
\mathcal{C}_r(a) \log (r+2) & \text{for }d(1-\beta)/2 = 2,
\\
\mathcal{C}_r(a) &\text{for }d(1-\beta)/2 > 2.
\end{cases}
\end{align*}
\end{lemma}

\section{Error representation of the two-scale expansion}
\label{DerivationFormulas}

The derivation of equation \eqref{ErrorTwoScale} which is satisfied by the two-scale expansion is standard and may e.\,g.\ be found in \cite{GNO4}. For the reader's convenience, we give a brief derivation here:
\begin{align*}
&-\nabla \cdot \Big(a \nabla \Big(u_{\shom}+\sum_i \phi_i \partial_i u_{\shom}\Big)\Big)
\\&
=-\nabla \cdot \Big(a\Big(\sum_i\phi_i \nabla \partial_i u_{\shom}\Big)\Big)-\nabla \cdot \Big(a \Big(\sum_i (e_i+\nabla \phi_i) \partial_i u_{\shom}\Big)\Big)
\\&
\overset{\eqref{EquationPhi}}{=}-\nabla \cdot \Big(a\Big(\sum_i\phi_i \nabla \partial_i u_{\shom}\Big)\Big)-\sum_i a (e_i+\nabla \phi_i) \cdot \nabla \partial_i u_{\shom}
\\&
\overset{\eqref{EquationSigma}}{=}-\nabla \cdot \Big(a\Big(\sum_i\phi_i \nabla \partial_i u_{\shom}\Big)\Big)-\sum_i a_{\shom} e_i \cdot \nabla \partial_i u_{\shom} - \sum_{i,j} (\nabla \cdot \sigma_{ij}) \, \partial_i \partial_j u_{\shom}
\\&\hspace{-0.6cm}
\overset{\sigma_{ijk}=-\sigma_{ikj}}{=}
-\nabla \cdot \Big(a\Big(\sum_i\phi_i \nabla \partial_i u_{\shom}\Big)\Big)-\nabla \cdot (a_{\shom}\nabla u_{\shom})
- \nabla \cdot \Big(\sum_{i,j} \sigma_{ij} \partial_i \partial_j u_{\shom}\Big)
\\&
=-\nabla \cdot (a_{\shom}\nabla u_{\shom})-\nabla \cdot \Big(a\Big(\sum_i\phi_i \nabla \partial_i u_{\shom}\Big)-\sum_{i} \sigma_{i}\nabla \partial_i u_{\shom}\Big).
\end{align*}
If the second-order term is included in the two-scale expansion, the previous computation implies together with \eqref{Equationpsi}, \eqref{EquationPsi}, and the skew-symmetry of $\sigma$ and $\Psi$
\begin{align*}
&-\nabla \cdot \Big(a \nabla \Big(u_{\shom}+\sum_i \phi_i \partial_i u_{\shom}+\sum_{i,j}\psi_{ij}\partial_i\partial_j u_{\shom}\Big)\Big)
\\&
=
-\nabla \cdot (a_{\shom}\nabla u_{\shom})-\nabla \cdot \Big(\sum_i (\phi_i a-\sigma_i) \nabla \partial_i u_{\shom}\Big)
-\nabla \cdot \Big(a \nabla \sum_{i,j}\psi_{ij}\partial_i\partial_j u_{\shom}\Big)
\\&
\stackrel{\eqref{Equationpsi}}{=}-\nabla \cdot (a_{\shom}\nabla u_{\shom})
-a:\sum_i\phi_i \nabla^2 \partial_i u_{\shom}
-\nabla \cdot \Big(a \sum_{i,j}\psi_{ij}\nabla \partial_i\partial_j u_{\shom}\Big)
-\sum_{i,j}\nabla \partial_i\partial_j u_{\shom}\cdot a\nabla \psi_{ij}
\\&
\stackrel{\eqref{EquationPsi}}{=}-\nabla \cdot (a_{\shom}\nabla u_{\shom})
-\sum_{i,j,k} \partial_i\partial_j \partial_k u_{\shom} \nabla \cdot \Psi_{ijk}
-\varepsilon \sum_{i,j,k} a_{\shom,ijk}^1 \partial_i\partial_j\partial_k u_{\shom}
-\nabla \cdot \Big(a \sum_{i,j}\psi_{ij}\nabla \partial_i\partial_j u_{\shom}\Big)
\\&
=-\nabla \cdot (a_{\shom}\nabla u_{\shom})
-\nabla \cdot \sum_{i,j,k} \partial_i\partial_j \partial_k u_{\shom} \Psi_{ijk}
-\varepsilon \sum_{i,j,k} a_{\shom,ijk}^1 \partial_i\partial_j\partial_k u_{\shom}
-\nabla \cdot \Big(a \sum_{i,j}\psi_{ij}\nabla \partial_i\partial_j u_{\shom}\Big).
\end{align*}
Using the definition
\begin{align*}
-\nabla \cdot \big(a_{\shom} \nabla u_{\shom}^1 \big) = 
\sum_{i,j,k} a_{\shom,ijk}^1 \partial_i\partial_j\partial_k u_{\shom},
\end{align*}
we deduce 
\begin{align*}
&-\nabla \cdot \Big(a \nabla \Big((u_{\shom}+\varepsilon u_{\shom}^1)+\sum_i \phi_i \partial_i (u_{\shom}+\varepsilon u_{\shom}^1)+\sum_{i,j}\psi_{ij}\partial_i\partial_j u_{\shom}\Big)\Big)
\\&
=-\nabla \cdot (a_{\shom}\nabla u_{\shom})
-\nabla \cdot \sum_{i,j,k} \partial_i\partial_j \partial_k u_{\shom} \Psi_{ijk}
-\nabla \cdot \Big(a \sum_{i,j}\psi_{ij}\nabla \partial_i\partial_j u_{\shom}\Big)
\\&\ \ \ 
-\nabla \cdot \Big(\varepsilon \sum_i (a\phi_i-\sigma_i) \nabla \partial_i u_{\shom}^1\Big).
\end{align*}

In the computation that showed $a_{\shom,ijk}^1=0$ for symmetric coefficient fields (i.\,e.\ the argument preceding \eqref{EquationPsi}), we have made use of ``integration by parts in expectations''. Let us provide details for this argument. Take some cutoff satisfying $\theta\equiv 1$ in $\{|x|\leq 1\}$ and $\theta\equiv 0$ outside of $\{|x|\leq 2\}$ as well as $|\nabla \theta| \leq C$. We then have by stationarity and qualitative ergodicity
\begin{align*}
\lim_{R\rightarrow\infty}
\frac{1}{\int \theta(x/R)\,dx}\int a\nabla \psi_{ij} \cdot e_k \theta(x/R) \,dx
=\mathbb{E}[a\nabla \psi_{ij} \cdot e_k]
\end{align*}
and
\begin{align*}
\lim_{R\rightarrow\infty}
\frac{1}{\int \theta(x/R)\,dx}\int a\nabla \psi_{ij} \cdot \nabla \phi_k \theta(x/R) \,dx
=\mathbb{E}[a\nabla \psi_{ij} \cdot \nabla \phi_k]
\end{align*}
almost surely. To see this, note that the above limits -- with $\lim$ replaced by either $\liminf$ or $\limsup$ -- define a shift-invariant random variable, which according to the definition of qualitative ergodicity must be almost surely constant. The identification of the limit is then a consequence of Fubini's theorem.

However, integration by parts (taking into account \eqref{EquationPhi}) yields
\begin{align}
\label{IntegrationByPartsExpectation}
&\frac{1}{\int \theta(x/R)\,dx}\int a\nabla \psi_{ij} \cdot e_k \theta(x/R) \,dx
+
\frac{1}{\int \theta(x/R)\,dx}\int a\nabla \psi_{ij} \cdot \nabla \phi_k \theta(x/R) \,dx
\\&
\nonumber
=
-\frac{1}{\int \theta(x/R)\,dx}\int \psi_{ij} (e_k+\nabla \phi_k) \cdot R^{-1} a(\nabla \theta)(x/R) \,dx.
\end{align}
Passing to the limit $R\rightarrow \infty$ and using the almost sure sublinear growth of $R\mapsto \big(\fint_{\{|x|\leq 2R\}} |\psi_{ij}|^2 \,dx\big)^{1/2}$ to deduce that the limit of the right-hand side vanishes, we infer
\begin{align*}
\mathbb{E}[a\nabla \psi_{ij} \cdot e_k]
=-\mathbb{E}[a\nabla \psi_{ij} \cdot \nabla \phi_k].
\end{align*}

\section{Proof of Theorem \ref{ErrorEstimateSymmetric} and Theorem \ref{ErrorEstimateNonsymmetric}: Translating corrector bounds into error estimates}

As Theorem \ref{ErrorEstimateSymmetric} is basically just a special case of Theorem~\ref{ErrorEstimateNonsymmetric}, the reason for $a_{\shom,ijk}^1$ vanishing in the case of symmetric coefficient fields having already been discussed in the introduction, we only need to prove Theorem~\ref{ErrorEstimateNonsymmetric}.

We recall the following estimate on stochastic moments of $r_\ast$ and the corrector $\phi$, $\sigma$ from \cite{GNO4}.
\begin{theorem}[Consequence of Theorem 1 and 3 in \cite{GNO4}]
\label{PhiStarBound}
Let $\langle\cdot\rangle$ be a stationary, bounded, and uniformly elliptic ensemble of coefficient fields on $\mathbb{R}^d$, $d\geq 3$, that satisfies a coarsened logarithmic Sobolev inequality of the form \eqref{LSI} with $\varepsilon:=1$ and exponent $\beta$. Suppose that we have $0\leq \beta<1-2/d$.

Then for any $\delta>0$ there exists a constant $C(d,\lambda,\beta)$ such that
\begin{align*}
r_\ast := \inf\bigg\{r\geq 1: \forall R\geq r \quad \frac{1}{R^2} \fint_{\{|x|\leq R\}} \Big|(\phi,\sigma)-\fint_{\{|x|\leq R\}} (\phi,\sigma) \Big|^2 \,dx \leq \delta
\bigg\}
\end{align*}
has stretched exponential moments in the sense that
\begin{align*}
\mathbb{E}\Big[\exp\big(r_\ast^{d(1-\beta)}/C(d,\lambda,\beta)\big)\Big]
\leq C(d,\lambda,\beta).
\end{align*}

Furthermore, the corrector $(\phi,\sigma)$ satisfies
\begin{align}
\label{EstimateOnCorrector}
\sup_{r\geq 1} \left(\fint_{\{|x|\leq r\}} |\phi|^2+|\sigma|^2 \,dx\right)^{1/2}
\leq
\mathcal{C}(a)
\end{align}
for a random constant $\mathcal{C}(a)$ that has stretched exponential moments in the sense
\begin{align*}
\mathbb{E}\Big[\exp\big(\mathcal{C}(a)^{1/C(d,\lambda,\beta)}\big)\Big]
\leq 2.
\end{align*}
\end{theorem}
To see that one may indeed take the supremum with respect to $r$ in the bound \eqref{EstimateOnCorrector}, one combines the original estimate from \cite[Theorem 3, estimate (50)]{GNO4}
\begin{align*}
\left(\fint_{\{|x-y|\leq 1\}} |\phi|^2+|\sigma|^2 \,dx\right)^{1/2}
\leq
\mathcal{C}(a,y)
\end{align*}
for some stationary random field $\mathcal{C}(a,x)$ having stretched exponential moments with the maximal ergodic theorem
\begin{align*}
\mathbb{E} \left[\sup_{r\geq 1} \fint_{\{|x|\leq r\}} |\mathcal{C}(a,x)|^p \,dx\right]
\leq C \mathbb{E} \left[|\mathcal{C}(a,x)|^{2p}\right]^{1/2}
\end{align*}
which holds for any stationary random field $\mathcal{C}$; furthermore, one uses the characterization that a random variable $\xi$ has stretched exponential moments if and only if there exists $m\in \mathbb{N}$ such that $\mathbb{E}[|\xi|^p]^{1/p} \leq p^m$ holds for all $p$ large enough.

\begin{proof}[Proof of Theorem~\ref{ErrorEstimateNonsymmetric}]

Having assumed that the ensemble satisfies a LSI with correlation length $\varepsilon$, by the rescaling
\begin{align*}
x\mapsto \tilde x:=\frac{x}{\varepsilon}
\end{align*}
we arrive at the setting of Theorem~\ref{SecondCorrectorEstimate}. Note that the first-order correctors $\phi$, $\sigma$ scale according to $\phi(x)=\varepsilon \tilde \phi(\tilde x)$ and $\sigma(x)=\varepsilon \tilde \sigma(\tilde x)$, where $\tilde \phi$ and $\tilde \sigma$ denote the corrector in microscopic coordinates; correspondingly, the second-order correctors $\psi$, $\Psi$ scale according to $\psi(x)=\varepsilon^2 \tilde \psi(\tilde x)$ and $\Psi(x)=\varepsilon^2 \tilde \Psi(\tilde x)$. By the choice of the scaling in \eqref{Defa1ijk}, the coefficient $a_{\shom,ijk}^1$ is invariant under rescaling.

Let us first establish the bound $|a_{\shom,ijk}^1|\leq C(d,\lambda,\beta)$. Taking a look at definition \eqref{Defa1ijk}, we notice that the bound is an immediate consequence of $\mathbb{E}[|\nabla \psi_{ij}|^2] \leq C(d,\lambda,\beta)\varepsilon^2$ and $\mathbb{E}[|\phi|^2+|\sigma|^2] \leq C(d,\lambda,\beta)\varepsilon^2$. Taking into account the scaling of $\phi$, $\sigma$, and $\nabla \psi_{ij}$, we infer these bounds from Proposition~\ref{proppsimoment} and Theorem~\ref{PhiStarBound} (which are stated for $\varepsilon=1$).

For the proof of the error estimates, we restrict our attention to the case $d(1-\beta)/2<2$; the other cases $d(1-\beta)/2=2$ and $d(1-\beta)/2>2$ are analogous.

We infer from Theorem~\ref{SecondCorrectorEstimate} and our rescaling
\begin{align}
\label{BoundPsi2m}
\Bigg(\fint_{\{|x|\leq 2^m \}}& \left|\psi-\fint_{\{|x|\leq 2^m \}}\psi\right|^2+\left|\Psi-\fint_{\{|x|\leq 2^m \}}\Psi\right|^2 \,dx \Bigg)^{1/2}
\\&\nonumber
\leq
\begin{cases}
\mathcal{C}_m(a) 2^{m(2-d(1-\beta)/2)} \varepsilon^{d(1-\beta)/2}
&\text{for }d(1-\beta)/2<2,
\\
\mathcal{C}_m(a) \varepsilon^{2} \sqrt{m-\log \varepsilon}
&\text{for }d(1-\beta)/2=2,
\\
\mathcal{C}_m(a) \varepsilon^{2}
&\text{for }d(1-\beta)/2>2,
\end{cases}
\end{align}
for all $m\in \mathbb{N}$ with the random constant $\mathcal{C}_m(a)$ satisfying
\begin{align*}
\mathbb{E}[\exp(\mathcal{C}_m(a)^{1/C(d,\lambda,\beta)})] \leq 2
\end{align*}
for some constant $C(d,\lambda,\beta)$.
Similarly, we obtain from Theorem~\ref{PhiStarBound}
\begin{align}
\label{BoundPhi2m}
\left(\fint_{\{|x|\leq 2^m \}} |\phi|^2+|\sigma|^2 \,dx \right)^{1/2}
\leq \hat{\bar{\mathcal{C}}}(a) \varepsilon
\end{align}
for all $m\in \mathbb{N}$ with some random constant $\hat{\bar{\mathcal{C}}}(a)$ satisfying
\begin{align*}
\mathbb{E}[\exp(\hat{\bar{\mathcal{C}}}(a)^{1/C(d,\lambda,\beta)})] \leq 2.
\end{align*}

Note that we may actually replace the averages $\fint_{\{|x|\leq 2^m \}}\psi$  in the estimate \eqref{BoundPsi2m} by a single value depending only on $a$ but not on $m$ (at the expense of replacing the constants $\mathcal{C}_m(a)$ by larger constants which also have stretched exponential moments; the same also applies to the averages $\fint_{\{|x|\leq 2^m \}}\Psi$): For $d(1-\beta)/2<2$, one may replace $\fint_{\{|x|\leq 2^m \}}\psi$ by $\fint_{\{|x|\leq 2^1\}} \psi$ for all $m$ and estimate the error for $m>1$ as 
\begin{align*}
\left|\fint_{\{|x|\leq 2^m\}} \psi-\fint_{\{|x|\leq 2^1\}} \psi\right|
&\leq
\sum_{n=2}^m \left|\fint_{\{|x|\leq 2^n \}}\psi-\fint_{\{|x|\leq 2^{n-1}\}}\psi\right|
\\&
\stackrel{\eqref{BoundPsi2m}}{\leq}  2^{m(2-d(1-\beta)/2)} \varepsilon^{d(1-\beta)/2} 
\sum_{n=2}^m C(d) \mathcal{C}_n(a) 2^{(n-m)(2-d(1-\beta)/2)},
\end{align*}
where the random constant arising from the latter sum certainly satisfies a bound analogous to the one of $\mathcal{C}_m(a)$. In contrast, for $d(1-\beta)/2>2$, one may either use an argument based on a similar telescoping sum and Proposition~\ref{MomentBoundsCorrectorAverages} to replace $\fint_{\{|x|\leq 2^m\}} \psi$ by $\lim_{n\rightarrow \infty} \fint_{\{|x|\leq 2^n\}} \psi$, thereby obtaining a bound analogous to \eqref{BoundPsi2m}, or use the same argument as before at the expense of incurring an additional factor of $\log 2^m=m$, which is easily canceled for the purpose of the remainder of the proof by the decay properties of $\nabla^3 u_{\shom}$, see below. The latter alternative of incurring a factor of $m$ also applies to the critical case $d(1-\beta)/2=2$.

As the equations \eqref{Equationpsi} and \eqref{EquationPsi} are invariant with respect to subtraction of constants from $\psi$ and $\Psi$, for the remainder of the proof we shall assume that for $\big(\fint_{\{|x|\leq 2^m \}} |\psi|^2 + |\Psi|^2 \,dx\big)^{1/2}$ a bound like \eqref{BoundPsi2m} holds.

By \eqref{ErrorTwoScaleSecond}, for a solution $u_{\shom}$ of the equation $-\nabla \cdot (a_{\shom}\nabla u_{\shom})=f$ and a solution $u$ of the equation $-\nabla \cdot (a\nabla u)=f$, the error in the second-order two-scale expansion satisfies
\begin{align*}
&-\nabla \cdot \Big(a\nabla \Big(u-(u_{\shom}+\varepsilon u_{\shom}^1)-\sum_i \phi_i (\partial_i u_{\shom}+\varepsilon \partial_i u_{\shom}^1) -\sum_{i,j} \psi_{ij} \partial_i\partial_j u_{\shom} \Big)\Big)
\\&
=\nabla \cdot \Big(\sum_{i,j} (\psi_{ij}a-\Psi_{ij}) \nabla \partial_i \partial_j u_{\shom}\Big)
+\nabla \cdot \Big(\sum_{i} \varepsilon (\phi_i a-\sigma_i) \nabla \partial_i u_{\shom}^1 \Big).
\end{align*}
The standard energy estimate for the elliptic equation satisfied by the error in the two-scale expansion, obtained by testing the equation with $u-(u_{\shom}+\varepsilon u_{\shom}^1)-\sum_i \phi_i (\partial_i u_{\shom}+\varepsilon \partial_i u_{\shom}^1) -\sum_{i,j} \psi_{ij} \partial_i\partial_j u_{\shom}$ and using ellipticity and boundedness of $a$, implies
\begin{align*}
&\left|\left|
\nabla u- \nabla \Big((u_{\shom}+\varepsilon u_{\shom}^1)+\sum_i \phi_i (\partial_i u_{\shom}+\varepsilon \partial_i u_{\shom}^1)+\sum_{i,j} \psi_{ij} \partial_i\partial_j u_{\shom}\Big)
\right|\right|_{L^2(\mathbb{R}^d)}
\\&
\leq C(d,\lambda) \big|\big|(|\psi|+|\Psi|)\nabla^3 u_{\shom}\big|\big|_{L^2(\mathbb{R}^d)}
+ C(d,\lambda) \big|\big|\varepsilon (|\phi|+|\sigma|)\nabla^2 u_{\shom}^1\big|\big|_{L^2(\mathbb{R}^d)}
\\&
\leq C(d,\lambda) \sum_{m=1}^\infty ||\nabla^3 u_{\shom}||_{L^\infty(\{ 2^{m-1}-1\leq |x|\leq 2^m\})}
\big(||\psi||_{L^2(\{|x|\leq 2^m\})}+||\Psi||_{L^2(\{|x|\leq 2^m\})}\big)
\\&\ \ \ 
+C(d,\lambda) \varepsilon \sum_{m=1}^\infty ||\nabla^2 u_{\shom}^1||_{L^\infty(\{ 2^{m-1}-1\leq |x|\leq 2^m\})}
\big(||\phi||_{L^2(\{|x|\leq 2^m\})}+||\sigma||_{L^2(\{|x|\leq 2^m\})}\big).
\end{align*}
Having assumed that $f$ is supported in $\{|x|\leq 1\}$, the solution $u_{\shom}$ to the constant-coefficient equation $-\nabla \cdot (a_{\shom}\nabla u_{\shom}) = f$ decays like $|x|^{2-d}$ for $|x|\in \mathbb{R}^d \setminus \{|x|\leq 2\}$. By regularity of constant-coefficient elliptic equations, its third derivative therefore decays like $|x|^{-1-d}$. In total, we deduce for all $x\in \mathbb{R}^d$
\begin{align*}
|\nabla^3 u_{\shom}(x)| \leq C(d,\lambda) |x|^{-d-1} ||\nabla^3 u_{\shom}||_{L^\infty}.
\end{align*}
A similar bound (with $|x|^{-d-2}$) holds for the fourth derivative.
From \eqref{EffectiveEquationFirst} and the bound on $a_{\shom}^1$, we therefore infer for all $x\in \mathbb{R}^d$ (see for example \cite[estimate (85)]{BellaGiuntiOtto})
\begin{align*}
|\nabla^2 u_{\shom}^1(x)| \leq C(d,\lambda,\beta) |x|^{-d-1} (||\nabla^2 u_{\shom}^1||_{L^\infty}+||\nabla^3 u_{\shom}||_{L^\infty}).
\end{align*}
Plugging in these estimates and the bounds \eqref{BoundPsi2m} and \eqref{BoundPhi2m} in the previous estimate for the error of the two-scale expansion, we deduce
\begin{align}
\nonumber
&\left|\left|
\nabla u- \nabla \Big(u_{\shom}+\varepsilon u_{\shom}^1+\sum_i \phi_i (\partial_i u_{\shom}+\varepsilon \partial_i u_{\shom}^1)+\sum_{i,j} \psi_{ij} \partial_i\partial_j u_{\shom}\Big)
\right|\right|_{L^2(\mathbb{R}^d)}
\\&
\label{TwoScaleErrorFirst}
\leq C(d,\lambda)\varepsilon^{d(1-\beta)/2} ||\nabla^3 u_{\shom}||_{L^\infty} \sum_{m=1}^\infty \mathcal{C}_m(a) 2^{m(1-d/2-d(1-\beta)/2)}
\\&\ \ \ \nonumber
+C(d,\lambda,\beta)\hat{\bar{\mathcal{C}}}(a) \varepsilon^2
(||\nabla^2 u_{\shom}^1||_{L^\infty}+||\nabla^3 u_{\shom}||_{L^\infty})
\sum_{m=1}^\infty 2^{-m(1+d/2)}.
\end{align}
Notice that the infinite sum $\sum_{m=1}^\infty \mathcal{C}_m(a) 2^{m(1-d/2-d(1-\beta)/2)}$ yields a random constant $\tilde{\mathcal{C}}(a)$ having stretched exponential moments in the sense
\begin{align*}
\mathbb{E}[\exp(\tilde{\mathcal{C}}(a)^{1/C(d,\lambda,\beta)})] \leq 2.
\end{align*}
This provides the desired error estimate at the level of the gradient for the second-order two-scale expansion.

By the Sobolev embedding $H^1(\mathbb{R}^d)\rightarrow L^{2^\ast}(\mathbb{R}^d)$ and the estimates \eqref{BoundPsi2m} and \eqref{BoundPhi2m} on $||\psi||_{L^2(B_R)}$ and $||\phi||_{L^2(B_R)}$, we now obtain from \eqref{TwoScaleErrorFirst}
\begin{align*}
&\left|\left|u-\Big(u_{\shom}+\varepsilon u_{\shom}^1+\sum_i \phi_i \partial_i u_{\shom}\Big)
\right|\right|_{L^2(\{|x|\leq 1\})}
\\&
\leq
C \left|\left|u-\Big(u_{\shom}+\varepsilon u_{\shom}^1+\sum_i \phi_i (\partial_i u_{\shom}+\varepsilon \partial_i u_{\shom}^1)+\sum_{i,j} \psi_{ij} \partial_i\partial_j u_{\shom}\Big)
\right|\right|_{L^{2^\ast}(\{|x|\leq 1\})}
\\&\ \ \ 
+\left|\left|
\sum_{i,j} \psi_{ij} \partial_i\partial_j u_{\shom}
\right|\right|_{L^2(\{|x|\leq 1\})}
+\left|\left|
\varepsilon \sum_i \phi_i \partial_i u_{\shom}^1
\right|\right|_{L^2(\{|x|\leq 1\})}
\\&
\leq C(d,\lambda) (||\nabla^3 u_{\shom}||_{L^\infty}+||\nabla^2 u_{\shom}||_{L^\infty(\{|x|\leq 1\})}) \tilde{\mathcal{C}}(a) \varepsilon^{d(1-\beta)/2}
\\&\ \ \ 
+C(d,\lambda,\beta) \big(||\nabla^2 u_{\shom}^1||_{L^\infty}+||\nabla^3 u_{\shom}||_{L^\infty}
+||\nabla u_{\shom}^1||_{L^\infty(\{|x|\leq 1\})}\big)\hat{\bar{\mathcal{C}}}(a) \varepsilon^2.
\end{align*}
Furthermore, going back to microscopic coordinates $\tilde x:=x/\varepsilon$, Lemma \ref{EstimateCorrectorH-1} provides for all $r\geq 1$ a bound of the form
\begin{align*}
r^{-d/2} ||\tilde \phi_i||_{H^{-1}(\{|\tilde x|\leq r\})} \leq
\begin{cases}
\hat{\mathcal{C}}_r(a) r^{2-d(1-\beta)/2}&\text{for }d(1-\beta)/2<2,
\\
\hat{\mathcal{C}}_r(a) \log (r+2)&\text{for }d(1-\beta)/2=2,
\\
\hat{\mathcal{C}}_r(a) &\text{for }d(1-\beta)/2>2,
\end{cases}
\end{align*}
with a random constant $\hat{\mathcal{C}}_r(a)$ subject to an estimate of the form
\begin{align*}
\mathbb{E}\big[\exp(\hat{\mathcal{C}}_r(a)^{1/C(d,\lambda,\beta)})\big]\leq 2.
\end{align*}
Reverting the change of variables, this implies in our setting
\begin{align*}
||\phi_i||_{H^{-1}(\{|x|\leq 1\})} \leq \hat{\mathcal{C}}(a) \varepsilon^{d(1-\beta)/2}.
\end{align*}

The embedding $L^2(\{|x|\leq 1\})\hookrightarrow H^{-1}(\{|x|\leq 1\})$ and the estimate
\begin{align*}
||\phi_i \partial_i u_{\shom}||_{H^{-1}(\{|x|\leq 1\})}\leq C(d) ||\phi_i||_{H^{-1}(\{|x|\leq 1\})} (||\nabla u_{\shom}||_{L^\infty(\{|x|\leq 1\})}+||\nabla^2 u_{\shom}||_{L^\infty(\{|x|\leq 1\})})
\end{align*}
then yield a bound of the form
\begin{align*}
&||u-(u_{\shom}+\varepsilon u_{\shom}^1)||_{H^{-1}(\{|x|\leq 1\})}
\\&
\leq C(d,\lambda) (\tilde{\mathcal{C}}(a)+\hat{\mathcal{C}}(a))
\big(||\nabla^3 u_{\shom}||_{L^\infty}+||\nabla^2 u_{\shom}||_{L^\infty(\{|x|\leq 1\})}
+||\nabla u_{\shom}||_{L^\infty(\{|x|\leq 1\})}\big) \varepsilon^{d(1-\beta)/2}
\\&
+ C(d,\lambda,\beta) \hat{\bar{\mathcal{C}}}(a)
\big(||\nabla^2 u_{\shom}^1||_{L^\infty}+||\nabla^3 u_{\shom}||_{L^\infty}+||\nabla u_{\shom}^1||_{L^\infty(\{|x|\leq 1\})}\big) \varepsilon^2.
\end{align*}
By the decay estimates for $\nabla^3 u_{\shom}$ and $\nabla^2 u_{\shom}$ outside of $\{|x|\leq 2\}$, we may estimate the norms $||\nabla^2 u_{\shom}||_{L^\infty(\{|x|\leq 1\})}$ and $||\nabla u_{\shom}||_{L^\infty(\{|x|\leq 1\})}$ in terms of $||\nabla^3 u_{\shom}||_{L^\infty}$; similarly, we may estimate $||\nabla u_{\shom}^1||_{L^\infty(\{|x|\leq 1\})}$ in terms of $||\nabla^2 u_{\shom}^1||_{L^\infty}$. This gives the desired estimate.
\end{proof}
In the proof of Theorem \ref{ErrorEstimateNonsymmetric}, we have used the estimate on the first-order corrector from Lemma \ref{EstimateCorrectorH-1}.

\begin{proof}[Proof of Lemma \ref{EstimateCorrectorH-1}]
Recall that a random variable $\xi$ has stretched exponential moments if and only if there exists $m\in \mathbb{N}$ such that $\mathbb{E}[|\xi|^p]^{1/p} \leq p^m$ holds for all $p$ large enough.

Using the notation of Lemma \ref{IntegrateLemma}, the result \cite[Theorem 2]{GNO4} provides the estimate
\begin{align}
\label{BoundCorrectorAverages}
\mathbb{E}\left[\left|\int \nabla (\phi,\sigma) \cdot g \right|^m\right]^{1/m}
\leq C(d,\lambda,\beta) m^{C(d,\lambda,\beta)} r^{-d(1-\beta)/2}
\end{align}
for the first-order correctors $\phi$ and $\sigma$ for all $m\geq 2$, all $r\geq 1$, all $x_0\in \mathbb{R}^d$, and all $g$ supported in $\{|x-x_0|\leq r\}$ with $(\fint_{\{|x-x_0|\leq r\}} |g|^{2+1/d} \,dx)^{1/(2+1/d)} \leq r^{-d}$. Correspondingly, we infer from Theorem~\ref{PhiStarBound} and the Caccioppoli inequality for $\phi$
\begin{align*}
\mathbb{E}\left[\left(\fint_{\{|x-x_0|\leq 1\}} |\nabla \phi|^2 \,dx\right)^{m/2}\right]^{1/m}
\leq C(d,\lambda,\beta) m^{C(d,\lambda,\beta)}
\quad\quad\text{for all }x_0\in \mathbb{R}^d
\end{align*}
for all $m\geq 2$ and (due to stationarity) all $x_0\in \mathbb{R}^d$.

Using these bounds as an input to Lemma \ref{IntegrateLemma}, we infer a bound of the form
\begin{align*}
\mathbb{E}\left[\left(r^{-d/2}\left|\left|\phi-\fint_{\{|x|\leq r\}}\phi\right|\right|_{H^{-1}(\{|x|\leq r\})} \right)^{m} \right]^{1/m}
\leq
\begin{cases}
C m^{C(d,\lambda,\beta)} r^{2-d(1-\beta)/2} &\text{for }d(1-\beta)/2<2,
\\
C m^{C(d,\lambda,\beta)} \log(r+2) &\text{for }d(1-\beta)/2=2,
\\
C m^{C(d,\lambda,\beta)} &\text{for }d(1-\beta)/2>2.
\end{cases}
\end{align*}
It only remains to drop the average $\fint_{\{|x|\leq r\}}\phi$. This is possible by exploiting the bound \eqref{BoundCorrectorAverages} on a sequence of dyadic balls to estimate
\begin{align*}
\left|\fint_{\{|x|\leq r\}}\phi\right|
&\leq \sum_{N=1}^\infty \left|\fint_{\{|x|\leq 2^{N-1} r\}}\phi-\fint_{\{|x|\leq 2^N r\}}\phi\right|
+\lim_{N\rightarrow \infty} \left|\fint_{\{|x|\leq 2^N r\}}\phi\right|.
\end{align*}
Note that $\fint_{\{|x|\leq 2^N r\}}\phi\rightarrow 0$ for $N\rightarrow \infty$ almost surely by ergodicity and our convention $\mathbb{E}[\phi]=0$. Therefore one obtains by \eqref{BoundCorrectorAverages}
\begin{align*}
\mathbb{E}\left[\left|\fint_{\{|x|\leq r\}}\phi\right|^m\right]^{1/m}
&\leq \sum_{N=1}^\infty \mathbb{E}\left[\left|\fint_{\{|x|\leq 2^{N-1} r\}}\phi-\fint_{\{|x|\leq 2^N r\}}\phi\right|^m\right]^{1/m}
\\&
\leq \sum_{N=1}^\infty C(d,\lambda,\beta) m^{C(d,\lambda,\beta)} (r 2^N)^{1-d(1-\beta)/2}
\leq C(d,\lambda,\beta) m^{C(d,\lambda,\beta)} r^{1-d(1-\beta)/2}.
\end{align*}
The precise details of the argument for dropping the average can be found in \cite[Lemma~2]{BellaFehrmanOtto}.
\end{proof}

\section{Proof of Proposition~\ref{prop1}: Estimate on the sensitivity of the second-order corrector}

\begin{proof}[Proof of Proposition~\ref{prop1}]
We decompose $g = (\tilde g, \bar g)$ so that 
\begin{equation*}
 F\nabla(\psi,\Psi) = \int \nabla \psi \cdot \tilde g + \int \nabla \Psi \cdot \bar g.
\end{equation*}
As in~\cite{GNO4} (see the discussion on page 13 there), in what follows we only consider ``well-localized' coefficient fields $a \in \Omega'$. Since $\Omega' \subset \Omega$ has full measure, w.l.o.g. we will omit $'$ in the notation. We split the proof into five steps.

\medskip

\step{1} Duality argument for $\psi$.\\
We denote by $(\tilde v_0, \tilde v_1, \tilde v_2, \tilde v_3)$ the unique quadruple of decaying solutions to
\begin{equation}\label{e1.2}
\begin{aligned}
  -\nabla \cdot (a^*\nabla \tilde v_0)&=\nabla \cdot \tilde g,
  \\
  -\nabla \cdot (a^*\nabla \tilde v_{1j})&= \nabla \tilde v_0 \cdot a e_j,
  \\
  -\Delta \tilde v_{2k} &= \partial_k \tilde v_0,
  \\
  -\nabla \cdot (a^* \nabla \tilde v_{3j}) &= - \nabla \cdot \Big(\sum_{k} \partial_j \tilde v_{2k} a^* e_k\Big) + \nabla \cdot \Big(\sum_{k} \partial_k \tilde v_{2k} a^* e_j\Big).
\end{aligned}
 \end{equation}
 for $j,k\in \{1,\ldots,d\}$.
For all $a \in \Omega'$ we denote by $\delta \phi, \delta \sigma, \delta \psi$ the (decaying) solutions of
\begin{subequations}
\begin{align}
\label{e1.1}
-\nabla\cdot (a\nabla \delta \phi_i)&=\nabla\cdot (\delta a (\nabla\phi_i+e_i))
\\
\label{e1.1b}
-\Delta \delta\sigma_{ijk}
&=\nabla\cdot \big[ (\delta a(\nabla\phi_i+e_i))_k e_j +(a \nabla \delta \phi_i)_k e_j
\\&\quad\quad\quad\nonumber
- (\delta a(\nabla\phi_i+e_i))_j e_k-(a\nabla \delta \phi_i)_je_k\big],
\\
\label{e1.1c}
-\nabla\cdot (a\nabla \delta \psi_{ij})&=\nabla\cdot [ \delta a \nabla \psi_{ij} + (\delta a\, \phi_i + a \, \delta \phi_i - \delta\sigma_i) e_j].
\end{align}
\end{subequations}
We define $\tilde F\nabla \psi:=\int \tilde g \cdot \nabla \psi_{ij}$, and use the shorthand notation $\delta \tilde F:=\int \tilde g\cdot \nabla \delta \psi_{ij}$. 
We shall prove that
\begin{equation}\label{e1-1}
\bigg|\bigg|\frac{\partial \tilde F\nabla \psi}{\partial a}\bigg|\bigg|_{CC}^2
\leq \sum_D \bigg(\int_D |\nabla \tilde v_0| (|\nabla \psi| + |\phi|) + \sum_{m=1}^3 |\nabla \tilde v_{m}| |\nabla \phi + e|) \bigg)^2.
\end{equation}
Indeed, by definition \eqref{e1.2} of $\tilde v_0$,
  \begin{equation}\label{e1.3}
\begin{aligned}
    \delta \tilde F & = \int \tilde g\cdot \nabla \delta \psi_{ij}
=-\int \nabla \tilde v_0 \cdot a \nabla \delta \psi_{ij}
\\
&= \int \nabla \tilde v_0 \cdot ( \delta a\, \nabla \psi_{ij} + \delta a\,\phi_i e_j ) + 
\int \nabla \tilde v_0 \cdot \delta \phi_i \, a e_j 
- \int \nabla \tilde v_0 \cdot \, \delta\sigma_i e_j , 
\end{aligned}
  \end{equation}
where the last relation follows from~\eqref{e1.1c}. Using~\eqref{e1.2} we rewrite the last two integrals in the following way:
\begin{align*}
 \int \nabla \tilde v_0 \cdot \delta \phi_i\, a e_j &= \int \nabla \tilde v_{1j} \cdot a \nabla \delta \phi_i \overset{\eqref{e1.1}}{=} -\int \nabla \tilde v_{1j} \cdot \delta a (\nabla \phi_i + e_i),
\end{align*}
and
\begin{align*}
- \int \nabla \tilde v_0 \cdot \delta\sigma_i e_j
&= -\int \sum_{k} \partial_k \tilde v_0 \, \delta \sigma_{ikj}
= -\int \sum_k \nabla \tilde v_{2k} \cdot \nabla \delta \sigma_{ikj}
\\
&\overset{\eqref{e1.1b}}{=} \sum_k \int 
\partial_j \tilde v_{2k} e_k\cdot (\delta a(\nabla \phi_i + e_i) + a\nabla \delta \phi_i)
- \partial_k \tilde v_{2k} e_j \cdot (\delta a(\nabla \phi_i + e_i) + a\nabla \delta \phi_i)
\\
&\overset{\eqref{e1.2}}{=} \int 
\sum_k (\partial_j \tilde v_{2k}e_k - \partial_k \tilde v_{2k}e_j) \cdot \delta a (\nabla \phi_i + e_i) 
+ \nabla \tilde v_{3j} \cdot a \nabla \delta \phi_i
\\
&\overset{\eqref{e1.1}}{=} \int 
\sum_k (\partial_j \tilde v_{2k}e_k - \partial_k \tilde v_{2k}e_j) \cdot \delta a (\nabla \phi_i + e_i) + \nabla \tilde v_{3j} \cdot \delta a (\nabla \phi_i + e_i).
\end{align*}
Plugging this into~\eqref{e1.3} yields
\begin{equation}\nonumber
 \begin{aligned}
    \delta \tilde F &= \int \nabla \tilde v_0 \cdot \delta a (\nabla \psi_{ij} + \phi_i e_j) + (-\nabla \tilde v_{1i} + \sum_k \partial_j \tilde v_{2k}e_k - \sum_k \partial_k \tilde v_{2k}e_j + \nabla \tilde v_{3i}) \cdot \delta a (\nabla \phi_i + e_i).
\end{aligned}
\end{equation}
Finally, using Definition~\ref{Definition1}, this implies~\eqref{e1-1}.

{
\providecommand{\vv}{\overline{\overline{v}}_0}

\medskip
\step{2} Duality argument for $\Psi$.\\
We consider $\bar F\nabla\Psi:=\int\nabla\Psi_{ijkl}\cdot\bar g$ and use the shorthand notation $\delta \bar F=\bar F\nabla \delta \Psi_{ijkl}=\int \bar g\cdot \nabla \delta \Psi_{ijkl}$. By $(\bar v_0, \ldots, \bar v_4)$ and $\overline{\overline{v}}_0$ 
we denote the decaying solutions of 
\begin{align}
 - \triangle \bar v_0 &= \nabla \cdot \bar g,
\label{e1.5.1}
\\
 - \nabla \cdot (a^* \nabla \bar v_{1ijkl}) &= \partial_k \bar v_0 ae_j \cdot e_l,
\label{e1.5.2}
\\ \nonumber
 - \nabla \cdot (a^* \nabla \overline{\overline{v}}_{0kl}) &= \nabla \cdot (\partial_k \bar v_0 a^* e_l ),
\\ \nonumber
  -\nabla \cdot (a^*\nabla \bar v_{2jkl}) &= \nabla \overline{\overline{v}}_{0kl} \cdot a e_j,
\\ \nonumber
  -\Delta \bar v_{3nkl} &= \partial_n \overline{\overline{v}}_{0kl},
\\ \nonumber
-\nabla \cdot (a^* \nabla \bar v_{4jkl}) &= - \nabla \cdot (\sum_{n} \partial_j \bar v_{3nkl} a^* e_n) + \nabla \cdot (\sum_{n} \partial_n \bar v_{3nkl} a^* e_j). 
\end{align}
We shall show that 
  \begin{equation}\label{e1-2}
    \left|\left|\frac{\partial \bar F\nabla \Psi}{\partial a}\right|\right|_{CC}^2
    \leq \sum_D \bigg(\int_D (|\nabla \bar v_0|+|\nabla \vv|) (|\nabla \psi| + |\phi|) + \sum_{m=1}^4 |\nabla \bar v_{m}| |\nabla \phi + e|) \bigg)^2.
  \end{equation}
By the skew-symmetry of $\sigma_{ijk}$, the definition of $\Psi$ (see \eqref{Gauge}) is equivalent to
\begin{align*}
-\Delta \Psi_{ijkl}
=
\nabla \cdot \big( &\operatorname{sym}_{ijl}[a\nabla \psi_{ij}\cdot e_l + \phi_i ae_j \cdot e_l-\varepsilon a_{\shom,ijl}^1] e_k
\\&
- \operatorname{sym}_{ijk}[a\nabla \psi_{ij}\cdot e_k + \phi_i ae_j \cdot e_k-\varepsilon a_{\shom,ijk}^1] e_l \big),
\end{align*}
which implies that
\begin{align}
\label{EquationdeltaPsi}
-\Delta \delta \Psi_{ijkl} =& 
\nabla \cdot \big(\operatorname{sym}_{ijl}[\delta a\,\nabla \psi_{ij}\cdot e_l + \phi_i \delta a\,e_j \cdot e_l] e_k
- \operatorname{sym}_{ijk}[\delta a\nabla \psi_{ij}\cdot e_k + \phi_i \delta a\,e_j \cdot e_k] e_l \big)
\\&
\nonumber
+\nabla \cdot \big(\operatorname{sym}_{ijl}[a \nabla \delta \psi_{ij}\cdot e_l + \delta \phi_i\, ae_j \cdot e_l] e_k
- \operatorname{sym}_{ijk}[a\nabla \delta \psi_{ij}\cdot e_k + \delta \phi_i \, a\,e_j \cdot e_k] e_l \big).
\end{align}
Denoting by $\operatorname{sym}_{ijk,l}$ the symmetrization with respect to $ijk$ and the skew-symmetrization with respect to $kl$, \eqref{e1.5.1} and \eqref{EquationdeltaPsi} imply
\begin{align}\label{e1.6.1}
 \delta \bar F = 2\operatorname{sym}_{ijk,l} \int \partial_k \bar v_0 (\delta a\,\nabla \psi_{ij}\cdot e_l + \phi_i \delta a\,e_j \cdot e_l+a \nabla \delta \psi_{ij}\cdot e_l + \delta \phi_i\, ae_j \cdot e_l).
\end{align}
We need to treat the terms on the right-hand side which do not include $\delta a$.
First, by~\eqref{e1.5.2} we have
\begin{equation}\label{e1.6.2}
 \int \delta \phi_i \partial_k \bar v_0 \, ae_j\cdot e_l = \int \nabla \bar v_{1ijkl} \cdot a \nabla \delta \phi_i \overset{\eqref{e1.1}}{=} \int \nabla \bar v_{1ijkl} \cdot  \delta a (\nabla \phi_i + e_i).
\end{equation}
To estimate the last remaining term we proceed as in Step 1:
\begin{align*}
\int \partial_k \bar v_0 a\nabla \delta \psi_{ij} \cdot e_l 
&= -\int \nabla \overline{\overline{v}}_{0kl} \cdot a \nabla \delta \psi_{ij}
\\
&= \int \nabla \overline{\overline{v}}_{0kl} \cdot [ \delta a\, \nabla \psi_{ij} + \delta a\,\phi_i e_j ] + 
\int \nabla \overline{\overline{v}}_{0kl} \cdot a \delta \phi_i\, e_j 
- \int \nabla  \overline{\overline{v}}_{0kl} \cdot \delta\sigma_i\, e_j, 
\end{align*}
and further by definition of auxiliary functions $\bar v_2, \bar v_3, \bar v_4$ by the same reasoning as in Step 1:
\begin{align*}
 \int \nabla \overline{\overline{v}}_{0kl} \cdot a \delta \phi_i\, e_j &= \int \nabla \bar v_{2jkl} \cdot a \nabla \delta \phi_i \overset{\eqref{e1.1}}{=} -\int \nabla \bar v_{2jkl} \cdot \delta a (\nabla \phi_i + e_i),
\\
 - \int \nabla \overline{\overline{v}}_{0kl} \cdot \delta\sigma_i\, e_j &=
\int \sum_n (\partial_j \bar v_{3nkl}e_n - \partial_n \bar v_{3nkl}e_j) \cdot \delta a (\nabla \phi_i + e_i) + \nabla \bar v_{4jkl} \cdot \delta a (\nabla \phi_i + e_i).
\end{align*}
Finally, we combine~\eqref{e1.6.1} and \eqref{e1.6.2} with the previous two equations and use Definition~\ref{Definition1} to obtain~\eqref{e1-2}. 
}

\medskip
\step{3} Consequences of the weighted Meyers-type estimate.\\
Using the bounds \eqref{e1.13} and \eqref{e1.13++}, we see that the definition \eqref{e1.2} implies that for any $1<p<\bar p$ ($\bar p$ being the Meyers' exponent), any $\alpha_1$ with $2p\leq \alpha_1<d(2p-1)$, and any $\alpha$ with $2p<\alpha<\alpha_1$ the following estimates hold:
\begin{align*}
\left(\int |\nabla \tilde v_0|^{2p}\omega_{\alpha} \right)^{\frac{1}{2p}}
&\lesssim \left(\int |\tilde g|^{2p}\omega_{\alpha_1}\right)^{\frac{1}{2p}},
\\
\left(\int |\nabla \tilde v_{1j}|^{2p}\omega_{\alpha-2p} \right)^{\frac{1}{2p}}
&\lesssim (r+r_*) \left(\int |\nabla \tilde v_0|^{2p}\omega_{\frac{1}{2}(\alpha+\alpha_1)}\right)^{\frac{1}{2p}}
\lesssim (r+r_*) \left(\int |\tilde g|^{2p}\omega_{\alpha_1}\right)^{\frac{1}{2p}},
\\
\left(\int |\nabla \tilde v_{2jkl}|^{2p}\omega_{\alpha-2p} \right)^{\frac{1}{2p}}
&\lesssim (r+r_*) \left(\int |\nabla \tilde v_0|^{2p}\omega_{\frac{1}{2}(\alpha+\alpha_1)}\right)^{\frac{1}{2p}}
\lesssim (r+r_*) \left(\int |\tilde g|^{2p}\omega_{\alpha_1}\right)^{\frac{1}{2p}},
\\
\left(\int |\nabla \tilde v_{3i}|^{2p}\omega_{\alpha-2p} \right)^{\frac{1}{2p}}
&\lesssim (r+r_*) \left(\int \sum_{j,k} |\nabla \tilde v_{2ijk}|^{2p}\omega_{\frac{1}{2}(\alpha+\alpha_1)-2p}\right)^{\frac{1}{2p}}
\lesssim (r+r_*) \left(\int |\tilde g|^{2p}\omega_{\alpha_1}\right)^{\frac{1}{2p}}.
\end{align*}
Note that by the definition of $r$ and $\omega_{\alpha_1}$, we have $\omega_{\alpha_1}\lesssim 1$ on the support of $g$; thus, we may drop the weights $\omega_{\alpha_1}$ in the integrals on the right. Similarly, we deduce the estimates
\begin{align*}
\left(\int |\nabla \bar v_0|^{2p}\omega_{\alpha} \right)^{\frac{1}{2p}}
&\lesssim \left(\int |\bar g|^{2p}\omega_{\alpha_1}\right)^{\frac{1}{2p}},
\\
\left(\int |\nabla \overline{\overline{v}}_0|^{2p}\omega_{\alpha} \right)^{\frac{1}{2p}}
&\lesssim 
\left(\int |\nabla \bar v_0|^{2p}\omega_{(\alpha+\alpha_1)/2} \right)^{\frac{1}{2p}}
\lesssim \left(\int |\bar g|^{2p}\omega_{\alpha_1}\right)^{\frac{1}{2p}},
\\
\left(\int |\nabla \bar {v}_m|^{2p}\omega_{\alpha-2p} \right)^{\frac{1}{2p}}
&\lesssim (r+r_*) \left(\int |\bar g|^{2p}\omega_{\alpha_1}\right)^{\frac{1}{2p}}
\quad\quad\text{for }m=1,2,3,4.
\end{align*}

\medskip
\step{4} Conclusion.\\
  First we observe that for any exponent $p>1$, any pair of functions $\varphi$ and $v$, and any positive weight function $w(x)$ we have by Cauchy-Schwarz' inequality followed by H\"older's inequality,
\begin{align*}
&\sum_D\left(\int_D|\nabla v||\varphi|\right)^2
\leq
\sum_D \left(|D|^{-\frac{p-1}{p}}\int_D|\nabla v|^2\omega^{\frac1p}\right)\left(|D|^{\frac{p-1}{p}}\int_D|\varphi|^2\omega^{-\frac1p}\right)
\\
&\qquad\leq \left(\sum_D|D|^{-(p-1)}\bigg(\int_D|\nabla v|^2\omega^{\frac1p}\bigg)^p\right)^{\frac1p}\left(\sum_D|D|\bigg(\int_D|\varphi|^2\omega^{-\frac1p}\bigg)^\frac{p}{p-1}\right)^{\frac{p-1}{p}},
\end{align*}
which combined with $\big(\int_D|\nabla v|^2\omega^{\frac1p}\big)^p\leq |D|^{p-1}\int_D|\nabla v|^{2p}\omega$ implies
\begin{align*}
\sum_D\left(\int_D|\nabla v||\varphi|\right)^2
\le \biggl( \int |\nabla v|^{2p}\omega \biggr)^{\frac1p} \biggl( \sum_D|D|(\min_D \omega)^{-\frac{1}{p-1}} \left(\int_D |\varphi|^2 \right)^\frac{p}{p-1}\biggr)^{\frac{p-1}{p}}.
\end{align*}
Let $\bar p$ denote the Meyers' exponent of Lemma~\ref{WeightedMeyers}. From~\eqref{e1-1} we deduce for exponents $1<p<\bar p$, $2p\leq\alpha<d(2p-1)$ and the weights $\omega$ defined by \eqref{DefWeights} the estimate
\begin{align*}
\left|\left|\frac{\partial F\nabla \psi}{\partial a}\right|\right|^2_{CC}
\lesssim&
(r+r_*)^2 \sum_{m=1}^3 \left(\int |\nabla \tilde v_m|^{2p}\omega_{\alpha-2p} \right)^{\frac{1}{p}}
\biggl(\sum_D|D| (\min_D\omega_{\alpha-2p})^{-\frac{1}{p-1}} \biggl( \int_D|\nabla\phi+e|^2 \biggr)^{\frac{p}{p-1}} \biggr)^{\frac{p-1}{p}} 
\\&
+\left(\int |\nabla \tilde v_0|^{2p} \omega_\alpha \right)^{\frac{1}{p}} 
\biggl( \sum_D|D| (\min_D\omega_{\alpha})^{-\frac{1}{p-1}} \biggl( \int_D|\nabla\psi|^2 + |\phi|^2 \biggr)^{\frac{p}{p-1}}\biggr)^{\frac{p-1}{p}},
\\
\lesssim&
(r+r_*)^2 \left(\int|g|^{2p}\right)^{\frac{1}{p}} 
\biggl(\sum_D|D| (\min_D\omega_{\alpha-2p})^{-\frac{1}{p-1}} \biggl( \int_D|\nabla\phi+e|^2 \biggr)^{\frac{p}{p-1}} \biggr)^{\frac{p-1}{p}} 
\\&
+\left(\int|g|^{2p}\right)^{\frac{1}{p}} 
\biggl( \sum_D|D| (\min_D\omega_{\alpha})^{-\frac{1}{p-1}} \biggl( \int_D|\nabla\psi|^2 + |\phi|^2 \biggr)^{\frac{p}{p-1}}\biggr)^{\frac{p-1}{p}},
\end{align*}
which completes the proof for $F\nabla \psi$. The conclusion for $F\nabla \Psi$ is entirely analogous, using \eqref{e1-2} instead of \eqref{e1-1}.
\end{proof}

\begin{proof}[Proof of Lemma~\ref{WeightedMeyers}]
\step{1} Dyadic estimates.\\
Let $u_0,u_1 \in H^1_{loc}(\R^d)$ and $g_0,g_1 \in L^2_{loc}(\R^d)$ be decaying and solving
\begin{equation}\nonumber
 -\nabla \cdot (a^* \nabla u_0) = \nabla \cdot g_0, \qquad -\nabla \cdot (a^* \nabla u_1) = g_1.
\end{equation}
For $\bar r\geq (r_*+1)$ we consider the dyadic decomposition $\R^d=\cup_{j\in\N_0}A_j$, 
  \begin{equation}\label{e1.11}
    A_j:=
    \begin{cases}
      \{\,|x|\leq 2\bar r\,\}&\text{for }j=0,\\
      \{\,2^{j}\bar r<|x|\leq 2^{j+1}\bar r\,\}&\text{for }j>0.
    \end{cases}
  \end{equation}
  We claim that for all $k,j\in\N_0$ we have
  \begin{align}\label{e1.10}  
    \biggl(\int_{A_j}|\nabla u_0|^2\biggr)^{\frac12} &\lesssim \sum_{k\in\N_0} \left(2^{-|k-j|d}\int_{A_k}|g_0|^2\right)^{\frac12},\\
    \label{e1.10+}
    \biggl(\int_{A_j}|\nabla u_1|^2\biggr)^{\frac12} &\lesssim \sum_{k\in\N_0}( (k-j)_++1) |A_0|^{\frac1d}\left(2^{-|k-j|d + 2\max\{k,j\}}\int_{A_k}|g_1|^2\right)^{\frac12},
  \end{align}
  where here and below $\lesssim$ means up to a constant only depending on $d$ and $\lambda$ (below, we shall also allow for dependence on $p$). The first relation is identical to~\cite[(161) in Step 2]{GNO4}, and the second is a consequence of the first: 

  Indeed, given $g_1 \in L^2_{loc}(\R^d)$, for $k \in \N_0$ we consider the decaying solution of $\triangle w_k = g_1 \chi_{A_k}$. Then using Green's function representation and H\"older's inequality we have, for $p \in (1,\infty)$ and for $j,k \in \N_0$ satisfying $|j-k| \ge 2$,
  \begin{equation}\nonumber
  \begin{aligned}
   \biggl(\int_{A_j}|\nabla w_k|^p\biggr)^{\frac1p} = & \biggl( \int_{A_j} \biggl| \int_{A_k} \nabla G(x,y) g_1(y) \dy \biggr|^p \dx \biggr)^{\frac1p} \\ 
\lesssim & \biggl( \int_{A_k} |g_1|^p \biggr)^{\frac1p}\biggl(\int_{A_j} \biggl( \int_{A_k} |\nabla G(x,y)|^{\frac{p}{p-1}} \dy \biggr)^{p-1} \dx \biggr)^{\frac1p}  
\\ \nonumber
\lesssim & |A_0|^{\frac1d} 2^{\frac1p(jd + kd(p-1) - (d-1)p\max(k,j))} \biggl( \int_{A_k} |g_1|^p \biggr)^{\frac1p}.
  \end{aligned}
  \end{equation}
  In the case $|j-k| \le 1$ the estimate follows immediately from the maximal regularity statement.
  
  Restricting to the case $p=2$ and using the triangle inequality, we see that the decaying solution $w$ of $\Delta w = g_1$ satisfies
  \begin{equation}\nonumber
   \biggl(\int_{A_j}|\nabla w|^2\biggr)^{\frac12} \lesssim \sum_{k\in\N_0} |A_0|^{\frac1d} 2^{\max\{k,j\}-|k-j|\frac d2}\left(\int_{A_k}|g_1|^2\right)^{\frac12}.
  \end{equation}
  Using~\eqref{e1.10} with $g_0 := \nabla w$ gives
  \begin{equation}\nonumber
   \begin{aligned}
   \biggl(\int_{A_j}|\nabla u_1|^2\biggr)^{\frac12} &\lesssim |A_0|^{\frac1d}\sum_{k\in\N_0} 2^{-|k-j|\frac{d}{2}}\sum_{k'\in\N_0} 2^{-|k-k'|\frac d2+ \max\{k,k'\}} \biggl(\int_{A_{k'}}|g_1|^2\biggr)^{\frac12}
\\ 
&= |A_0|^{\frac1d}\sum_{k'\in \N_0} \biggl(\int_{A_{k'}}|g_1|^2\biggr)^{\frac12} \sum_{k\in\N_0} 2^{-(|k-k'|+|k-j|)\frac{d}{2}+\max\{k,k'\}},
 \end{aligned}
  \end{equation}
and so \eqref{e1.10+} follows from the elementary bound
\begin{equation*}
 \begin{aligned}
 &\sum_{k\in\N_0} 2^{-(|k-k'|+|k-j|)\frac{d}{2}+\max\{k,k'\}}
 \\
 &= \sum_{k \le \min\{k',j\}} \ldots + \sum_{\min\{k',j\} < k < \max\{k',j\}} \ldots + \sum_{k \ge \max\{k',j\}} \ldots
 \\
 &\lesssim 2^{-|k'-j|\frac d2 + k'} + ((k'-j)_++1)2^{-|k'-j|\frac d2 + \max\{k',j\}} + 2^{-|k'-j|\frac d2 + \max\{k',j\}} 
 \\
  &\lesssim ( (k'-j)_++1)2^{-|k'-j|\frac d2 + \max\{k',j\}}.
 \end{aligned}
\end{equation*}

\medskip
\step{2} A Meyers-type estimate with weights.\\
In this step we establish the weighted Meyers estimates \eqref{e1.13}, \eqref{e1.13++} with the family of weight functions 
  \begin{equation}\nonumber
    \omega_\gamma(x)=\left(\frac{|x|}{r_*+r}+1\right)^\gamma.
  \end{equation}
  %
As already observed in~\cite[Proposition~3]{GNO4}, for $g$ supported in $\{|x|\leq r\}$ and $0 \le \alpha_0 < d(2p-1)$ we have
\begin{equation}\nonumber
  \left(\int |\nabla v|^{2p}\omega_{\alpha_0} \right)^{\frac{1}{2p}}\lesssim \left(\int |g|^{2p}\right)^{\frac{1}{2p}}.
\end{equation}
%
%
We start with the argument for~\eqref{e1.13}. Let $\{A_j\}_{j\in\N_0}$ denote the dyadic decomposition defined in~\eqref{e1.11} with $\bar r:=r_*+r$. From the classical Meyers estimate we deduce that for all $p\in[1,\bar p)$ and $j\in\N_0$ we have
\begin{equation}\label{e1.14}
    \left(\fint_{A_j}|\nabla v|^{2p}\right)^{\frac{1}{2p}}\lesssim \left(\fint_{A_j^+}|\nabla v|^{2}\right)^{\frac{1}{2}}+    \left(\fint_{A_j^+}|g|^{2p}\right)^{\frac{1}{2p}},
\end{equation}
where $A_j^+$ denotes the enlarged annulus defined by
\begin{equation*}
    A_j^+:=
    \begin{cases}
      A_0\cup A_1&\text{for }j=0,\\
      A_{j-1}\cup A_j\cup A_{j+1}&\text{for }j>0.
    \end{cases}
\end{equation*}
%
  %
  Since by~\eqref{e1.14}
  \begin{equation}\nonumber
   \int |\nabla {v}|^{2p}\omega_{\alpha_0} \lesssim \int |g|^{2p} \omega_{\alpha_0} + 
   \sum_{j\in\N_0}(\max_{A_j}\omega_{\alpha_0})|A_j|^{1-p}\biggl( \int_{A_j}|\nabla v|^{2} \biggr)^p,
  \end{equation}
  it is enough to show that
  \begin{equation}\label{e1.15}
   \sum_{j\in\N_0}(\max_{A_j}\omega_{\alpha_0})|A_j|^{1-p}\biggl( \int_{A_j}|\nabla v|^{2} \biggr)^p \lesssim \int |g|^{2p} \omega_{\alpha_1}.
  \end{equation}
Appealing to~\eqref{e1.10} in Step~3 with $g_0=g$ to estimate each integral, we get that 
\begin{align*}
   \sum_{j\in\N_0}(\max_{A_j}\omega_{\alpha_0})&|A_j|^{1-p}\biggl( \int_{A_j}|\nabla v|^{2} \biggr)^p 
   \lesssim 
   \sum_{j\in\N_0} 2^{j\alpha_0}|A_j|^{1-p} \biggl( \sum_{k\in\N_0} 2^{-|k-j|\frac{d}{2}}\biggl(\int_{A_k}|g|^2\biggr)^{\frac12} \biggr)^{2p}
   \\
   &\lesssim 
   \sum_{j\in\N_0} 2^{j\alpha_0}|A_j|^{1-p} \biggl( \sum_{k\in\N_0} 2^{-|k-j|\frac{d}{2}}|A_k|^\frac{p-1}{2p}\biggl(\int_{A_k}|g|^{2p}\biggr)^{\frac{1}{2p}} \biggr)^{2p}
   \\
   &\lesssim 
   \sum_{j\in\N_0} 2^{j\alpha_0}|A_j|^{1-p} \biggl( \sum_{k\in\N_0} 2^{-|k-j|\frac{d}{2}}|A_k|^\frac{p-1}{2p}2^{-k\frac{\alpha_1}{2p}}\biggl(\int_{A_k}|g|^{2p}\omega_{\alpha_1}\biggr)^{\frac{1}{2p}} \biggr)^{2p}
   \\
   &\lesssim 
   \sum_{j\in\N_0} 2^{j\alpha_0}2^{jd(1-p)} \biggl( \sum_{k\in\N_0} 2^{-|k-j|\frac{d}{2}}2^{kd\frac{p-1}{2p}
   }2^{-k\frac{\alpha_1}{2p}}\biggl(\int_{A_k}|g|^{2p}\omega_{\alpha_1}\biggr)^{\frac{1}{2p}} \biggr)^{2p}
   \\
   &\lesssim 
   \sum_{j\in\N_0} 2^{j\alpha_0}2^{jd(1-p)} \biggl( \sum_{k\in\N_0} 2^{-|k-j|\frac{dp}{2p-1}}2^{kd\frac{p-1}{2p-1}
   }2^{-k\frac{\alpha_1}{2p-1}}\biggr)^{2p-1} 
   \biggl(\sum_{k\in\N_0} \int_{A_k}|g|^{2p}\omega_{\alpha_1}\biggr),
  \end{align*}
  and~\eqref{e1.15} follows since 
  \begin{align}\label{e1.16}
   \sum_{k\in\N_0} 2^{-|k-j|\frac{dp}{2p-1}}2^{kd\frac{p-1}{2p-1}
   }2^{-k\frac{\alpha_1}{2p-1}} \lesssim 2^{(jd\frac{p-1}{2p-1} - j\frac{\alpha_1}{2p-1})},
\end{align}
and therefore  
\begin{align}\nonumber
   \sum_{j\in\N_0} 2^{j\alpha_0}2^{jd(1-p)} \biggl( \sum_{k\in\N_0} 2^{-|k-j|\frac{dp}{2p-1}}2^{kd\frac{p-1}{2p-1}
   }2^{-k\frac{\alpha_1}{2p-1}}\biggr)^{2p-1} 
   \lesssim 
   \sum_{j\in\N_0} 2^{j\alpha_0}2^{jd(1-p)} 2^{(jd(p-1) - j\alpha_1)} \lesssim 1,
\end{align}
where the sum converges thanks to $\alpha_1 > \alpha_0$. When $g_0$ is supported only in $A_0$ (i.e., the special case considered in~\cite{GNO4}), the sum in~\eqref{e1.16} is replaced by the term for $k=0$, and consequently the summability is guaranteed by the condition $ 0 \le \alpha_0 < d(2p-1)$. 
  
To show~\eqref{e1.13++}, we use the Meyer's estimate for solutions of equation with right-hand side in non-divergence form of the form~\eqref{e1.17b}:
\begin{equation}\label{e1.18}
    \left(\fint_{A_j}|\nabla \hat v|^{2p}\right)^{\frac{1}{2p}}\lesssim |A_j|^{\frac1d} \left(\fint_{A_j^+}|\hat g|^{2p}\right)^{\frac{1}{2p}} + \left(\fint_{A_j^+}|\nabla \hat v|^{2}\right)^{\frac{1}{2}},
\end{equation}
which follows from the standard Meyer's estimate.

Since $|A_j|^{\frac1d} = |A_0|^{\frac1d} 2^j$ and $2^j \lesssim \frac{|x|}{r+r_*}+1$ for $x\in A_j$, we have by~\eqref{e1.18}
\begin{equation}\nonumber
   \int |\nabla \hat{v}|^{2p}\omega_{\alpha_0} \lesssim |A_0|^{\frac{2p}{d}} \int |\hat g|^{2p} \omega_{\alpha_0+2p} + 
   \sum_{j\in\N_0}(\max_{A_j}\omega_{\alpha_0})|A_j|^{1-p}\biggl( \int_{A_j}|\nabla \hat v|^{2} \biggr)^p,
\end{equation}
and so it is enough to show that
\begin{equation}\label{e1.19}
   \sum_{j\in\N_0}(\max_{A_j}\omega_{\alpha_0})|A_j|^{1-p}\biggl( \int_{A_j}|\nabla \hat v|^{2} \biggr)^p \lesssim |A_0|^{\frac{2p}{d}}\int |\hat g|^{2p} \omega_{\alpha_1}
\end{equation}
with $\alpha_1 > \alpha_0 + 2p$. 
Appealing to~\eqref{e1.10+} in Step 3, we see that 
\begin{align*}
   \sum_{j\in\N_0}&(\max_{A_j}\omega_{\alpha_0})|A_j|^{1-p}\biggl( \int_{A_j}|\nabla \hat v|^{2} \biggr)^p 
   \\
   &\lesssim 
   |A_0|^{1-p} \sum_{j\in\N_0} 2^{j\alpha_0}2^{jd(1-p)} \biggl( \sum_{k\in\N_0}( (k-j)_++1) |A_0|^{\frac1d} 2^{-|k-j|\frac d2 + \max\{k,j\}} \biggl( \int_{A_k}|\hat g|^2\biggr)^{\frac12} \biggr)^{2p}
   \\
   &\lesssim 
   |A_0|^{1-p} \sum_{j\in\N_0} 2^{j\alpha_0}2^{jd(1-p)} \biggl( \sum_{k\in\N_0}( (k-j)_++1) |A_0|^{\frac1d} 2^{-|k-j|\frac d2 + \max\{k,j\}} |A_k|^{\frac{p-1}{2p}}\biggl( \int_{A_k}|\hat g|^{2p}\biggr)^{\frac{1}{2p}} \biggr)^{2p}
   \\
   &\lesssim 
   |A_0|^{\frac{2p}{d}} \sum_{j\in\N_0} 2^{j\alpha_0}2^{jd(1-p)} \biggl( \sum_{k\in\N_0}( (k-j)_++1) 2^{-|k-j|\frac d2 + \max\{k,j\}} 2^{kd\frac{p-1}{2p}} 2^{-k\frac{\alpha_1}{2p}}\biggl( \int_{A_k}|\hat g|^{2p} \omega_{\alpha_1} \biggr)^{\frac{1}{2p}} \biggr)^{2p}
   \\
   &\lesssim 
   |A_0|^{\frac{2p}{d}} \biggl( \int |\hat g|^{2p} \omega_{\alpha_1} \biggr) \sum_{j\in\N_0} 2^{j\alpha_0}2^{jd(1-p)} 
   \biggl( \sum_{k\in\N_0} \left( ( (k-j)_++1) 2^{-|k-j|\frac d2 + \max\{k,j\}} 2^{kd\frac{p-1}{2p}} 2^{-k\frac{\alpha_1}{2p}} \right)^{\frac{2p}{2p-1}} \biggr)^{2p-1},   
  \end{align*}
  which implies~\eqref{e1.19} provided we show the double sum is of order $1$. Since $\alpha_1 < d(2p-1)$, we see that 
  \begin{equation}\nonumber
   \biggl( \sum_{k\in\N_0} \left( ( (k-j)_++1) 2^{-|k-j|\frac d2 + \max\{k,j\}} 2^{kd\frac{p-1}{2p}} 2^{-k\frac{\alpha_1}{2p}} \right)^{\frac{2p}{2p-1}} \biggr)^{2p-1} 
   \lesssim 2^{2pj + jd(p-1) - j\alpha_1},
  \end{equation}
  which implies
  \begin{align*}
  \sum_{j\in\N_0} 2^{j\alpha_0}2^{jd(1-p)} &
   \biggl( \sum_{k\in\N_0} \left( ( (k-j)_++1) 2^{-|k-j|\frac d2 + \max\{k,j\}} 2^{kd\frac{p-1}{2p}} 2^{-k\frac{\alpha_1}{2p}} \right)^{\frac{2p}{2p-1}} \biggr)^{2p-1} 
   \\
   &\lesssim 
   \sum_{j\in\N_0} 2^{j\alpha_0}2^{jd(1-p)} 2^{2pj + jd(p-1) - j\alpha_1} \lesssim \sum_{j\in\N_0} 2^{j(\alpha_0 + 2p - \alpha_1)} \lesssim 1,
  \end{align*}
  the last estimate following from the fact that $\alpha_1 > \alpha_0 + 2p$. This concludes the proof.
\end{proof}

\section{Proof of Proposition \ref{StochasticMomentsSensitivity}: Estimate on stochastic moments of the sensitivity of the second-order corrector}

\begin{proof}[Proof of Proposition \ref{StochasticMomentsSensitivity}]
Raising both sides to the $m/2$-th power in the estimate \eqref{s.S2b} and taking the expectation, we infer by $\Big(\fint_{\{|x|\leq r\}} |g|^{2p}\Big)^\frac{1}{2p}\,\leq \, r^{-d}$
\begin{align*}
&\mathbb{E}\left[\left|\left|
\frac{\partial F\nabla(\psi,\Psi)}{\partial a}
\right|\right|_{CC}^m\right]
\\&
\leq C^m r^{-(2-1/p)dm/2}
\mathbb{E}\left[\left(\sum_D (r_* + r)^{\frac{2p}{p-1}} |D| (\min_D\omega_{\alpha-2p})^{-\frac{1}{p-1}} \biggl( \int_D|\nabla\phi+e|^2 \biggr)^{\frac{p}{p-1}} \right)^{m(p-1)/2p}
\right]
\\&\quad
+C^m r^{-(2-1/p)dm/2} \mathbb{E}\left[\left(\sum_D|D|
(\min_D\omega_{\alpha})^{-\frac{1}{p-1}} \biggl( \int_D|\nabla\psi|^2 + |\phi|^2 \biggr)^{\frac{p}{p-1}}\right)^{m(p-1)/2p}\right].
\end{align*}
Taking into account the bound $\omega_\gamma(x)\geq (1+r_\ast)^{-\gamma} (|x|/r+1)^\gamma$ for $\gamma>0$ (to see that, recall that we have assumed $r\geq 1$), we deduce by applying H\"older's inequality to the expectation
\begin{align*}
&\mathbb{E}\left[\left|\left|
\frac{\partial F\nabla(\psi,\Psi)}{\partial a}
\right|\right|_{CC}^m\right]
\\&
\leq C^m r^{-(2-1/p)dm/2}
\mathbb{E}\left[(1+r_\ast)^{2m+m(\alpha-2p)/p}\right]^{1/2}
\\&\quad\ \ \ \times
\mathbb{E}\left[\left(\sum_D (1 + r/r_*)^{\frac{2p}{p-1}} |D| (\min_D |x|/r+1)^{-(\alpha-2p)/(p-1)} \biggl( \int_D|\nabla\phi+e|^2 \biggr)^{\frac{p}{p-1}} \right)^{m(p-1)/p}\right]^{1/2}
\\&\quad
+C^m r^{-(2-1/p)dm/2} \mathbb{E}\left[(1+r_\ast)^{m\alpha/p}\right]^{1/2}
\\&\quad\ \ \ \times
\mathbb{E}\left[\left(\sum_D|D|
(\min_D |x|/r+1)^{-\alpha/(p-1)} \biggl( \int_D|\nabla\psi|^2 + |\phi|^2 \biggr)^{\frac{p}{p-1}}\right)^{m(p-1)/p}\right]^{1/2}.
\end{align*}
This entails using $|D|\lesssim r^{\beta d} (\min_D |x|/r+1)^{\beta d}$ (see Definition~\ref{Definition1}) and $r_\ast \geq 1$
\begin{align*}
&\mathbb{E}\left[\left|\left|
\frac{\partial F\nabla(\psi,\Psi)}{\partial a}
\right|\right|_{CC}^m\right]^{1/m}
\\&
\leq C r^{-(2-1/p)d/2+1+d(p-1)/2p+\beta d/2}
\mathbb{E}\left[(1+r_\ast)^{m\alpha/p}\right]^{1/2m}
\\&\ \ \ \times
\mathbb{E}\left[\left(r^{-d}\sum_D |D| (\min_D |x|/r+1)^{-(\alpha-2p)/(p-1)} \biggl( (\min_D |x|/r+1)^{\beta d} \fint_D|\nabla\phi+e|^2 \biggr)^{\frac{p}{p-1}} \right)^{m(p-1)/p}\right]^{1/2m}
\\&
+C r^{-(2-1/p)d/2+d(p-1)/2p+\beta d/2} \mathbb{E}\left[(1+r_\ast)^{m\alpha/p}\right]^{1/2m}
\\&\ \ \ \times
\mathbb{E}\left[\left(r^{-d}\sum_D|D|
(\min_D |x|/r+1)^{-\alpha/(p-1)} \biggl((\min_D |x|/r+1)^{\beta d} \fint_D|\nabla\psi|^2 + |\phi|^2 \biggr)^{\frac{p}{p-1}}\right)^{m(p-1)/p}\right]^{1/2m}.
\end{align*}
We next pull out the factors $(\min_D|x|/r+1)^{\beta d}$ from the inner parentheses. Our intention is to apply Jensen's inequality to the resulting weighted sums, the weights being of the form $r^{-d}|D| (\min_D |x|/r+1)^{-\gamma/(p-1)+\beta d p/(p-1)}$ with $\gamma=\alpha-2p$ and $\gamma=\alpha$. To apply Jensen in order to pull the power $m(p-1)/p$ under the sum, we need $m(p-1)/p>1$ (which is true for $m$ large) and a uniform bound on the sum of the weights. Owing to the factor $|D|$, the sum of the weights behaves like a Riemann sum for the integral of $r^{-d} (\min_D |x|/r+1)^{-\gamma/(p-1)+\beta d p/(p-1)}$. The sum of the weights is therefore bounded by a constant $C(d,\beta)$, provided that $-\gamma/(p-1)+\beta d p/(p-1)<-d$, which is satisfied provided that $-\alpha+2p+\beta d p<-d(p-1)$. In the case $\beta<1-2/d$, we may achieve this inequality by choosing $p>1$ close enough to $1$ and $\alpha<d$ close enough to $d$.

Thus, applying Jensen's inequality and just simplifying the prefactor involving $r$, we infer
\begin{align*}
&\mathbb{E}\left[\left|\left|
\frac{\partial F\nabla(\psi,\Psi)}{\partial a}
\right|\right|_{CC}^m\right]^{1/m}
\\&
\leq C r^{-(1-\beta)d/2+1}
\mathbb{E}\left[(1+r_\ast)^{m\alpha/p}\right]^{1/2m}
\\&\quad\ \ \ \times
\mathbb{E}\left[\sum_D r^{-d} |D| (\min_D |x|/r+1)^{-(\alpha-2p)/(p-1)+d\beta p/(p-1)} \biggl(\fint_D|\nabla\phi+e|^2 \biggr)^{m} \right]^{1/2m}
\\&\quad
+C r^{-(1-\beta)d/2} \mathbb{E}\left[(1+r_\ast)^{m\alpha/p}\right]^{1/2m}
\\&\quad\ \ \ \times
\mathbb{E}\left[\sum_D r^{-d} |D|
(\min_D |x|/r+1)^{-\alpha/(p-1)+d\beta p/(p-1)} \biggl(\fint_D|\nabla\psi|^2 + |\phi|^2 \biggr)^{m}\right]^{1/2m}.
\end{align*}
Using
\begin{align*}
\mathbb{E} \left[\biggl(\fint_D|\nabla\psi|^2 + |\phi|^2 \biggr)^{m}\right]
\lesssim \mathbb{E} \left[\biggl(\fint_{\{|x|\leq 1\}}|\nabla\psi|^2 + |\phi|^2 \biggr)^{m}\right]
\end{align*}
and
\begin{align*}
\mathbb{E} \left[\biggl(\fint_D|\nabla\phi+e|^2 \biggr)^{m}\right]
\lesssim \mathbb{E} \left[\biggl(\fint_{\{|x|\leq 1\}}|\nabla\phi+e|^2 \biggr)^{m}\right]
\end{align*}
(these inequalities follow by covering $D$ by balls, applying H\"older's inequality, and using stationarity of the involved quantities) as well as again the fact that the sum of the weights is bounded by a constant, we finally obtain
\begin{align*}
&\mathbb{E}\left[\left|\left|
\frac{\partial F\nabla(\psi,\Psi)}{\partial a}
\right|\right|_{CC}^m\right]^{1/m}
\\&
\leq C r^{-(1-\beta)d/2+1}
\mathbb{E}\left[(1+r_\ast)^{m\alpha}\right]^{1/2m}
\mathbb{E}\left[\biggl(\fint_{\{|x|\leq 1\}}|\nabla\phi+e|^2 \biggr)^{m} \right]^{1/2m}
\\&\quad
+C r^{-(1-\beta)d/2} \mathbb{E}\left[(1+r_\ast)^{m\alpha/p}\right]^{1/2m}
\mathbb{E}\left[\biggl(\fint_{\{|x|\leq 1\}}|\nabla\psi|^2 + |\phi|^2 \biggr)^{m}\right]^{1/2m}.
\end{align*}
This is the desired result.
\end{proof}

\section{Proof of Proposition~\ref{propCZ}: A large-scale $L^p$ theory for elliptic operators with random coefficients}

\begin{proof}[Proof of Proposition \ref{propCZ}]
 Since the statement and the proof of Proposition~\ref{propCZ} are almost identical to~\cite[Proposition 4]{DuerinckxGloriaOtto}, we will only point out the differencies. The proof follows the standard approach to Calder\'on-Zygmund estimate in the constant-coefficient case, that passes via a BMO-estimate and interpolation (see, e.g.,~\cite[Section 7.1.1]{giaqmart}).

 More precisely, combination of (A.26), (A.24), and (A.5) from that paper, together with the fact that the constant in the interpolation between $L^2 \to L^2$ and $L^\infty \to \textrm{BMO}$ grows like $p$ (this can be seen, e.g., by estimating constants in proofs of~\cite[Theorem 6.27, Theorem 6.29]{giaqmart}), we arrive at 
 \begin{equation}\nonumber
  \biggl( \int \biggl( \fint_{B_{\underline{r_*(x)}}(x)} |\nabla u|^2 \biggr)^p \,dx \biggr)^{\frac 1p} \le C(d,\lambda)p \biggl( \int \biggl( \sup_{R \ge 1} \fint_{B_R(x)} |g|^2 \biggr)^p \,dx \biggr)^{\frac 1p}.
 \end{equation}
 For $h(x) := \fint_{B_1(x)} |g|^2$ we observe that for any $R \ge 1$
 \begin{equation}\nonumber
  \int_{B_R(x)} h(y) \dy = \int_{B_R(x)} \fint_{B_1(y)} |g(z)|^2 \dz \,dy = \frac{1}{|B_1|}\int_{B_{R+1}(x)} |g(z)|^2 |\{ B_1(z) \cap B_R(x)\}| \dz \gtrsim \int_{B_R(x)} |g|^2,
 \end{equation}
 which together with boundedness of the maximal function as an operator from $L^p \to L^p$ implies (note that the bound $C(d)p$ on the operator norm of the maximal operator follows from the proof by duality, the $L^1-L^1_{weak}$ bound, the $L^2-L^2$ bound, and the behavior of the constant in the Marcinkiewicz interpolation theorem \cite[Proof of Theorem 2.4]{FourierAnalysis})
\begin{align*}
\int \biggl( \sup_{R \ge 1} \fint_{B_R(x)} |g|^2 \biggr)^p &\leq C(d)^p \int \biggl( \sup_{R \ge 1} \fint_{B_R(x)} h \biggr)^p \le C(d)^p \int \biggl( \sup_{R > 0} \fint_{B_R(x)} h \biggr)^p
\\
&\leq C(d)^p p^p \int h^p = C(d)^p p^p \int \biggl( \fint_{B_1(x)} |g|^2 \biggr)^p \dx,
\end{align*}
and the proof of the proposition is complete.
\end{proof}

{
\providecommand{\Norm}[2]{\biggl( \int \biggl( \fint_{B_1(x)} | {#1} |^2 \biggr)^{\frac{#2}{2}}\dx \biggr)^{\frac {1}{#2}}}
\providecommand{\Normm}[3]{\biggl( \int \biggl( \fint_{B_{#3}(x)} | {#1} |^2 \biggr)^{\frac{#2}{2}}\dx \biggr)^{\frac {1}{#2}}}
\providecommand{\Normmm}[3]{\biggl( \int \biggl( \fint_{B_{#3}(x)} | {#1} |^2 \biggr)^{\frac{#2}{2}}\dx \biggr)^{\frac {2}{#2}}}
\providecommand{\NorM}[2]{\biggl( \int \biggl( \fint_{B_1(x)} | {#1} |^2 \biggr)^{\frac{#2}{2}}\dx \biggr)^{\frac {2}{#2}}}

\begin{lemma}\label{lm4}
 For smooth decaying $f$, let $u$ be the decaying solution of 
 \begin{equation}\label{e2.1}
  -\Laplace u = f.
 \end{equation}
 Then for $p \ge 2$ and $\frac{1}{p^*} = \frac 1p - \frac 1d$ we have
 \begin{equation}\label{e2.2}
  \Norm{\nabla u}{p^*} \lesssim (p^*)^2 \Norm{f}{p},
 \end{equation}
 where $\lesssim$ means $\le C$ for a constant only depending on $d$.
\end{lemma}

\begin{proof}
 The existence of a unique decaying solution $u$ of ~\eqref{e2.1} follows from the standard theory. 

Let $R \ge 1$, the value of which will be chosen later. Using the maximal regularity statement for the Laplace operator~\eqref{e2.1} of the form $\int |\nabla^2 u|^2 \lesssim \int |f|^2$, applied to $\eta (u - \bar u)$, with $\eta$ being a cutoff with $\eta\equiv 1$ in $B_{2R}$ and $\eta \equiv 0$ outside of $B_{3R}$ and with $\bar u = \fint_{B_{3R}} u$, we see
\begin{equation}\nonumber
\| \nabla^2 (\eta (u-\bar u))\|_{L^2}^2 \lesssim \| \eta f \|_{L^2}^2 + \|\nabla \eta \nabla u\|_{L^2}^2 + \| \Laplace \eta (u-\bar u)\|_{L^2}^2.
\end{equation}
Using Poincar\'e inequality on the last term together with properties of the cutoff function $\eta$ implies
\begin{equation}\label{e2.4}
\fint_{B_{2R}}  |\nabla^2 u|^2 \lesssim \fint_{B_{3R}} |f|^2 + \frac{1}{R^2} \fint_{B_{3R}} |\nabla u|^2. 
\end{equation}
Denoting $u_1$ a convolution of $u$ with a smooth mollifier $\varphi \ge 0$ such that $\supp \varphi \subset B_1$ and $\int \varphi = 1$, we see
\begin{equation}\label{e2.5}
\fint_{B_R} |\nabla u_1 - \nabla u|^2 \lesssim \fint_{B_{2R}} |\nabla^2 u|^2. 
\end{equation}
We combine~\eqref{e2.4} and~\eqref{e2.5} to arrive at
\begin{equation}\nonumber
\fint_{B_R} |\nabla u|^2 \lesssim \fint_{B_R} |\nabla u_1|^2 + \fint_{B_{3R}} |f|^2 + \frac{1}{R^2} \fint_{B_{3R}}  |\nabla u|^2. 
\end{equation}
Using this estimate for balls centered at $x$, and applying the $L^{\frac{p^\ast}{2}}(\R^d)$-norm to both sides, we get by the triangle inequality
 \begin{equation}\label{e2.7}
 \begin{aligned}
  \Normmm{\nabla u}{p^*}{R} &\lesssim \Normmm{\nabla u_1}{p^*}{R} + \Normmm{f}{p^*}{3R} 
  \\
   &\quad + \frac{1}{R^2} \Normmm{\nabla u}{p^*}{3R}.
 \end{aligned}
 \end{equation} 
 Covering $B_{3R}$ with finitely many balls $B_R$ (the number of balls being independent of $R$), we see by the triangle inequality
 \begin{equation}\nonumber
  \Normmm{\nabla u}{p^*}{3R} \le C(d) \Normmm{\nabla u}{p^*}{R}.
 \end{equation}

 Hence we can fix large enough $R=R(d)$ such that the last term in~\eqref{e2.7} can be absorbed to the left-hand side:
 \begin{equation}\nonumber
 \begin{aligned}
  \Normmm{\nabla u}{p^*}{R} &\lesssim \Normmm{\nabla u_1}{p^*}{R} + \Normmm{f}{p^*}{3R}.
 \end{aligned}
 \end{equation}
 Using again a covering argument, at the expense of additional $C(d)$ we see that in fact
 \begin{equation}\label{e2.11}
 \begin{aligned}
  \Normmm{\nabla u}{p^*}{1} &\lesssim \Normmm{\nabla u_1}{p^*}{1} + \Normmm{f}{p^*}{1}.
 \end{aligned}
 \end{equation}
 
 Next we observe that for any measurable function $g$ and any exponents $2 \le p_1 \le p_0 < \infty$ we have reverse H\"older inequality of the form
 \begin{equation}\label{e2.12}
  \NorM{g}{p_0} \le C(d) \NorM{g}{p_1}.
 \end{equation}
 Indeed, if $\Sigma$ denotes the set of all cubes $Q$ with side length $1$ and integer coordinates, then we see that 
 \begin{equation}\nonumber
  \int_{B_1(x)} |g|^2 \le \sum_{Q \cap B_1(x) \neq \emptyset} a_Q^2
 \end{equation}
 where $a_Q^2 = \fint_Q |g|^2$, which after applying the triangle inequality in $L^{\frac{p_0}{2}}(\R^d)$ turns into
 \begin{equation}\label{e2.13}
  \NorM{g}{p_0} \le C(d) \biggl( \sum_{Q \in \Sigma} a_Q^{p_0} \biggr)^{\frac{2}{p_0}}.
 \end{equation}
 At the same time, since $|\{ x \in \R^d : Q \subset B_1(x) \}|$ is independent of $Q$ and positive, we see that also  
 \begin{equation}\label{e2.14}
  \biggl( \sum_{Q \in \Sigma} a_Q^{p_1} \biggr)^{\frac{2}{p_1}} \le C(d) \NorM{g}{p_1}.
 \end{equation}
 Since for sequences the $l^{p_0}$ norm is bounded by the $l^{p_1}$ norm, \eqref{e2.12} immediately follows from \eqref{e2.13} and~\eqref{e2.14}.

Using the fact that Poisson equation is linear, in particular $- \Laplace u_1 = f_1$ so that the maximal regularity statement applies, by the Jensen and Sobolev inequalities (for the dependence of the Sobolev constant on $p^*$, see for example \cite{SobolevConst}) we have 
 \begin{equation}\label{e2.15}
 \begin{aligned}
  \Norm{\nabla u_1}{p^*} &\le \biggl( \int |\nabla u_1|^{p^*} \biggr)^{\frac{1}{p^*}} \lesssim p^* \biggl( \int |\nabla^2 u_1|^p \biggr)^{\frac 1p} \lesssim (p^*)^2 \biggl( \int |f_1|^p \biggr)^{\frac 1p} 
\\ &\lesssim (p^*)^2 \Norm{f}{p}.
  \end{aligned}
 \end{equation}
To conclude, we plug~\eqref{e2.15} into~\eqref{e2.11}, and use~\eqref{e2.12} with $g=f$ to obtain~\eqref{e2.2}.
\end{proof}
}

%


{
\providecommand{\Norm}[2]{\biggl( \int \biggl( \fint_{B_1(x)} | {#1} |^2 \biggr)^{\frac{#2}{2}}\dx \biggr)^{\frac {1}{#2}}}
\providecommand{\Normm}[3]{\biggl( \int \biggl( \fint_{B_{#3}(x)} | {#1} |^2 \biggr)^{\frac{#2}{2}}\dx \biggr)^{\frac {1}{#2}}}
\providecommand{\rr}{{\underline{r_*}}}
\providecommand{\NormR}[4]{\biggl( \int_{B_{#3}} \biggl( \fint_{B_{#4}(x)} | {#1} |^2 \biggr)^{\frac{#2}{2}}\dx \biggr)^{\frac {1}{#2}}}
\providecommand{\Normmm}[3]{\biggl( \int \biggl( \fint_{B_{#3}(x)} | {#1} |^2 \biggr)^{\frac{#2}{2}}\dx \biggr)^{\frac {2}{#2}}}

\begin{lemma}\label{lmiteration}
Let the assumptions of Proposition~\ref{propCZ} be satisfied.
 Let $\nabla u$ and $g$ be stationary vector fields, related through
 \begin{equation}\nonumber
  -\nabla \cdot (a \nabla u) = \nabla \cdot g.
 \end{equation}
 Then for any $p \in [2,d)$ we have
 \begin{equation}\label{e4.0}
  \mathbb{E}\Biggl[ (\rr)^{-\frac{dp^*}{2}} \biggl( \int_{B_1} |\nabla u|^2 \biggr)^{\frac{p^*}{2}} \Biggr]^\frac{1}{p^*}
  \lesssim 
  (p^*)^3 \mathbb{E}\Biggl[ \biggl( \int_{B_1} |g|^2 \biggr)^{\frac{p^*}{2}} \Biggr]^\frac{1}{p^*} + 
  (p^*)^3 \mathbb{E}\Biggl[ \biggl( \int_{B_1} |\nabla u|^2 \biggr)^{\frac{p}{2}} \Biggr]^\frac 1p.
 \end{equation}
 where $\frac{1}{p^*} = \frac 1p - \frac 1d$, and where $\lesssim$ means $\le C$ for a constant only depending on $d$ and $\lambda$.
\end{lemma}

\begin{proof}
 By ergodicity of $\en{\cdot}$ it is enough to show that for any radius $R \ge 1$ and a generic realization of the coefficient field $a$ we have
 \begin{multline}\label{e4.1}
   \biggl( \int_{B_{\frac R2}} (\rr(x))^{-\frac{dp^*}{2}} \biggl( \int_{B_1(x)} |\nabla u|^2 \biggr)^\frac{p^*}{2} \dx \biggr)^{\frac{1}{p^*}} 
   \\ \lesssim (p^*)^3 \NormR{g}{p^*}{2R}{1} + \frac{(p^*)^3}{R} \NormR{\nabla u}{p}{3R}{1}.
 \end{multline}

 For $R \ge 1$ and $\eta_R := \eta(\cdot/R)$, with $\eta$ being a smooth cut-off function with $\eta\equiv 1$ in $B_1$ and $\eta \equiv 0$ outside of $B_2$, and a constant $\bar u$ to be chosen later, the function $u_R := \eta_R (u - \bar u)$ solves
 \begin{equation}\nonumber
  -\nabla \cdot (a \nabla u_R) = \nabla \cdot (\eta_R g - a \nabla \eta_R (u-\bar u)) - \nabla \eta_R \cdot (g + a\nabla u).
 \end{equation} 
  
 By Lemma~\ref{lm4}, the vector field $h := \nabla w$, with $-\triangle w = \nabla \eta_R \cdot (g + a\nabla u)$ satisfies the estimate
 \begin{equation}\label{p.1}
  \Norm{h}{p^*} \lesssim (p^*)^2 \Norm{\nabla \eta_R \cdot (g + a\nabla u)}{p}
 \end{equation}
 and clearly
 \begin{equation}\nonumber
  -\nabla \cdot (a \nabla u_R) = \nabla \cdot (\eta_R g + h - a \nabla \eta_R (u-\bar u)),
 \end{equation}
 which by Proposition~\ref{propCZ} (applied with the exponent $p^*/2$) and~\eqref{p.1} yield
 \begin{equation}\nonumber
 \begin{aligned}
  \Normm{\nabla u_R}{p^*}{\rr} \lesssim & p^* \Norm{\eta_R g}{p^*} \\ + & p^* \Norm{\nabla \eta_R (u-\bar u)}{p^*} \\ + & (p^*)^3 \Norm{\nabla \eta_R \cdot g}{p} \\ + & (p^*)^3 \Norm{\nabla \eta_R \cdot a \nabla u}{p},
 \end{aligned}
 \end{equation}
 the $p^*$-factor from~\eqref{propCZ} being $p^*$ and not $(p^*)^2$ since we have taken the square root of the estimate.
 We will now estimate all four integrals on the right-hand side.

 By the support condition on $\eta$ we have for the first and the third integral
 \begin{align*}
  \Norm{\eta_R g}{p^*} &\lesssim 
  \biggl( \int_{B_{2R}} \biggl( \fint_{B_1(x)} |g|^2 \biggr)^{\frac{p^*}{2}} \,dx \biggr)^{\frac{1}{p^*}},\\
  \Norm{\nabla \eta_R \cdot g}{p} &\lesssim \frac 1R \biggl( \int_{B_{2R}} \biggl( \fint_{B_1(x)} |g|^2 \biggr)^{\frac{p}{2}} \,dx \biggr)^{\frac{1}{p}}
  \overset{\textrm{H\"older}}{\lesssim} \biggl( \int_{B_{2R}} \biggl( \fint_{B_1(x)} |g|^2 \biggr)^{\frac{p^*}{2}} \,dx \biggr)^{\frac{1}{p^*}}.
 \end{align*}
 For the fourth integral we similarly have 
 \begin{equation*}
  \Norm{\nabla \eta_R \cdot a \nabla u}{p} \lesssim \frac 1R \biggl( \int_{B_{2R}} \biggl( \fint_{B_1(x)} |\nabla u|^2 \biggr)^{\frac{p}{2}}  \,dx \biggr)^{\frac{1}{p}}.
 \end{equation*}
 For the second integral we will show that 
 \begin{equation}\nonumber
  \Norm{\nabla \eta_R (u-\bar u)}{p^*} \lesssim \frac{p^*}{R} \biggl( \int_{B_{3R}} \biggl( \fint_{B_1(x)} |\nabla u|^2\biggr)^{\frac{p}{2}} \,dx \biggr)^{\frac{1}{p}},
 \end{equation}
 which by properties of $\eta_R$ follows from
 \begin{equation}\label{e4.2}
  \NormR{u-\bar u}{p^*}{2R}{1} \lesssim p^* \NormR{\nabla u}{p}{3R}{1}.
 \end{equation}
 To this end we consider $u_1 := u * \varphi$, with $\varphi$ being a smooth convolution kernel on scale $1$ (i.\,e.\ $\varphi \in C^\infty$, $\varphi \ge 0$, $\supp \varphi \in B_1$, $\int \varphi = 1$). Then 
 \begin{equation*}
  \int_{B_1(x)} |u(y) - u_1(y)|^2 \dy \lesssim \int_{B_2(x)} |\nabla u|^2,
 \end{equation*}
 and by covering $B_2$ by finitely many balls of radius $1$ we get
 \begin{equation}\nonumber
  \NormR{u - u_1}{p^*}{2R}{1} \lesssim \NormR{\nabla u}{p^*}{3R}{1} \lesssim \NormR{\nabla u}{p}{3R}{1},
 \end{equation}
 where the last relation follows from~\eqref{e2.12}. Using the triangle inequality, \eqref{e4.2} follows from
 \begin{align*}
  \NormR{u_1 - \bar u}{p^*}{2R}{1} &\overset{\textrm{Jensen}}{\lesssim} \biggl( \int_{B_{2R}} |u_1 - \bar u|^{p^*} \biggr)^{\frac{1}{p^*}} \overset{\textrm{Sobolev}}{\lesssim} p^* \biggl( \int_{B_{2R}} |\nabla u_1|^p \biggr)^{\frac 1p} 
\\
  &\overset{\textrm{Jensen}}{\lesssim} p^* \NormR{\nabla u}{p}{3R}{1},
 \end{align*}
 where $\bar u$ was chosen such that $\int_{B_{2R}} (u_1 - \bar u) = 0$. For the dependence of the Sobolev constant on $p^*$ that was used here, see for example \cite{SobolevConst}.

 Collecting all the estimates, we have 
 \begin{equation}\label{e4.5}
  \Normm{\nabla u_R}{p^*}{\rr} \lesssim (p^*)^3 \NormR{g}{p^*}{2R}{1} + \frac{(p^*)^3}{R} \NormR{\nabla u}{p}{3R}{1}.
 \end{equation}
For $x \in B_{R/2}$ we have 
 \begin{equation}\nonumber
  \int_{B_1(x)} |\nabla u|^2 = \int_{B_1(x)} |\nabla u_R|^2 \le \int_{B_{\rr}(x)} |\nabla u_R|^2 \lesssim (\rr)^{d} \fint_{B_{\rr}(x)} |\nabla u_R|^2 
 \end{equation}
 where the first relation comes from the fact that $\nabla u = \nabla u_R$ in $B_R$. 
Integrating the $(p^*/2)$-th power of this relation over $B_{R/2}$, together with~\eqref{e4.5} yields~\eqref{e4.1}.
\end{proof}
}

\section{Proof of Proposition \ref{proppsimoment}: Moment bounds for the gradient of the second-order corrector}

\providecommand{\psimom}[1]{\mathbb{E}\Bigg[ \biggl( \fint_{\{|x|\leq 1\}} |\nabla \psi_{ij}|^2 \biggr)^{\frac{#1}{2}} \Bigg]^{\frac{1}{#1}}}

\begin{proof}[Proof of Proposition \ref{proppsimoment}]
 For any $q \in [2,\infty)$, Theorem~\ref{PhiStarBound} implies
 \begin{equation}\label{e5.1}
  \mathbb{E}\big[(\underline{r_*}^d)^q \big]^\frac 1q \le \mathbb{E}\big[ (r_*^d)^q \big]^\frac 1q \lesssim q^\frac{1}{1-\beta},
  \qquad 
  \mathbb{E}\Bigg[\biggl( \fint_{\{|x|\leq 1\}} |\phi|^2+|\sigma|^2 \biggr)^\frac q2 \Bigg]^\frac 1q \lesssim q^{C(d,\lambda,\beta)}.
 \end{equation}
 Focusing first on $\nabla \psi_{ij}$, we can iteratively use Lemma~\ref{lmiteration}, since $\psi_{ij}$ satisfies
 \begin{equation}\nonumber
  -\nabla \cdot (a \nabla \psi_{ij}) = \nabla \cdot ( (a\phi_i - \sigma_i) e_j).
 \end{equation}
 Indeed we see that 
 for any $q < d $ and any $\alpha \in (1,\infty)$ we get from Lemma~\ref{lmiteration} using~\eqref{e5.1} 
 \begin{equation}
 \begin{aligned}\label{p.2}
  &\mathbb{E}\Bigg[\biggl( \fint_{\{|x|\leq 1\}} |\nabla \psi_{ij}|^2 \biggr)^{\frac{q^*}{2\alpha}} \Bigg]^{\frac{\alpha}{q^*}}
  \\
  &\overset{\textrm{H\"older}}\le \mathbb{E}\Big[ (\underline{r_*}^d)^\frac{q^*}{2(\alpha-1)} \Big]^\frac{\alpha-1}{q^*} 
  \mathbb{E}\Bigg[ (\underline{r_*}^d)^{-\frac{q^*}{2}} \biggl( \fint_{\{|x|\leq 1\}} |\nabla \psi_{ij}|^2 \biggr)^{\frac{q^*}{2}}\Bigg]^{\frac{1}{q^*}}
 \\ 
&\overset{\eqref{e5.1},\eqref{e4.0}}{\lesssim} (q^*)^3 \p{ \frac{q^*}{\alpha -1 }}^\frac{1}{2(1-\beta)} 
  \biggl( \mathbb{E}\bigg[ \biggl( \fint_{\{|x|\leq 1\}} |\phi|^2+|\sigma|^2 \biggr)^{\frac{q^*}{2}} \bigg]^\frac{1}{q^*} + \psimom{q}
 \biggr)
\\
&\overset{\eqref{e5.1}}{\lesssim} (q^*)^3  \p{ \frac{q^*}{\alpha -1 }}^\frac{1}{2(1-\beta)} 
 \biggl( (q^*)^{C(d,\lambda,\beta)} + \psimom{q} \biggr).
\end{aligned}\end{equation}
Since the construction of $\nabla \psi_{ij}$ as a stationary field (analogous to the construction of $\nabla \phi$ and $\nabla \sigma$ in \cite{GNO4}) provides $\mathbb{E}[|\nabla \psi_{ij}|^2] \lesssim 1$, we can iterate the above estimate few times to show that 
\begin{equation}\nonumber
\mathbb{E}\Bigg[ \biggl( \fint_{\{|x|\leq 1\}} |\nabla \psi_{ij}|^2 \biggr)^{\frac{d}{2}} \Bigg]^{\frac 1d} \le C(d,\lambda,\beta).
\end{equation}
Finally, given any $p \in (d,\infty)$, we choose $q = \frac{d(p+1)}{d+p+1}$, $\alpha = \frac{p+1}{p}$ so that $q^*  = p+1$ and $\frac{q^*}{\alpha} = p$, and use~\eqref{p.2} to obtain 
 \begin{equation}\nonumber
  \mathbb{E}\Bigg[ \biggl( \fint_{\{|x|\leq 1\}} |\nabla \psi_{ij}|^2 \biggr)^{\frac{p}{2}} \Bigg]^{\frac{1}{p}} \le C(d,\lambda,\beta) p^{C(d,\lambda,\beta)}.
 \end{equation}
 
 Regarding the estimate for the flux corrector, we observe that by~\eqref{EquationPsi} it satisfies an equation with the same structure as equation for $\psi_{ij}$ (in fact even simpler since the operator has constant coefficients), and so the above argument applies verbatim also to~\eqref{EquationPsi}.
\end{proof}

\section{Proof of Proposition \ref{MomentBoundsCorrectorAverages}: Estimate on stochastic moments of averages of the second-order corrector}

\begin{proof}[Proof of Proposition \ref{MomentBoundsCorrectorAverages}]
Estimating the right-hand side of the bound in Proposition \ref{StochasticMomentsSensitivity} by means of Proposition \ref{proppsimoment}, Theorem \ref{PhiStarBound}, and the Caccioppoli inequality $\int_{\{|x|\leq 1\}} |\nabla \phi_i+e_i|^2 \,dx\leq C(\lambda)\int_{\{|x|\leq 2\}} |\phi_i+x_i|^2 \,dx$, we deduce
\begin{align*}
\left(\mathbb{E}~\Big|\Big|\frac{\partial F\nabla(\psi,\Psi)}{\partial a}\Big|\Big|_{CC}^m \right)^{1/m}
\leq
C(d,\lambda) m^{C(d,\beta,\lambda)} r^{-(1-\beta)d/2+1}.
\end{align*}
The $L^p$-version of the LSI in Lemma \ref{LpLSI} now implies the desired result; note that by the vanishing expectation of $\nabla \psi_{ij}$ and $\nabla \Psi_{ijkl}$, the expectation of $F\nabla (\psi,\Psi)$ also vanishes.
\end{proof}

\section{Proof of Lemma \ref{IntegrateLemma} and Theorem \ref{SecondCorrectorEstimate}: Translating estimates on spatial averages of the second-order corrector into estimates on norms}

\begin{proof}[Proof of Lemma \ref{IntegrateLemma}]
Let $v\in H^1$ be a function supported in the cube $[-2^N,2^N]^d$. Denote by $P_n:L^2([-2^N,2^N]^d)\rightarrow L^2([-2^N,2^N]^d)$ the projection operator that to any function $u$ associates a function $P_n u$ which coincides on all subcubes of $[-2^N,2^N]^d$ of the form $(2m_1\cdot 2^n,(2m_1+2)\cdot 2^n) \times \ldots \times (2m_d\cdot 2^n,(2m_d+2)\cdot 2^n)$ (with $m\in \mathbb{Z}^d$) with the average of $u$ on this cube (in particular, $P_n u$ is constant on each such cube). Note that besides the usual properties of projection operators $P_n P_n=P_n$ and $P_n^\ast = P_n$, we have $P_{n}P_{n-1}=P_{n}$ and $P_{n-1}P_n=P_n$. With this notation and these properties, we have for any $u\in L^2_{loc}(\mathbb{R}^d)$
\begin{align*}
&\int_{[-2^N,2^N]^d} (u-P_N u) v \,dx
\\&
=\sum_{n=1}^N \int_{[-2^N,2^N]^d} (P_{n-1}u-P_n u) v \,dx
+\int_{[-2^N,2^N]^d} (u-P_0 u) v\,dx
\\&
=\sum_{n=1}^N \int_{[-2^N,2^N]^d} (P_{n-1}u-P_n u) (P_{n-1} v-P_n v) \,dx
+\int_{[-2^N,2^N]^d} (u-P_0 u) (v-P_0 v)\,dx.
\end{align*}
We therefore obtain by H\"older's inequality
\begin{align}\nonumber
\left|\int_{[-2^N,2^N]^d} (u-P_N u) v \,dx\right|
\leq &\sum_{n=1}^N ||P_{n-1}u-P_n u||_{L^2([-2^N,2^N]^d)} ||P_{n-1}v-P_n v||_{L^2}
\\&
\nonumber
+ ||u-P_0 u||_{L^2([-2^N,2^N]^d)} ||v-P_0 v||_{L^2}.
\end{align}
Note that we have $||P_{n-1}v-P_n v||_{L^2}^2\leq ||v-P_n v||_{L^2}^2$. This entails by Poincar\'e's inequality (applied to $v-P_n v$ on the cubes of side length $2\cdot 2^n$ as well as to $u-P_0 u$ on the cubes of side length $2^0$)
\begin{align*}
&\left|\int_{[-2^N,2^N]^d} (u-P_N u) v \,dx\right|
\lesssim \left(\sum_{n=1}^N 2^n ||P_{n-1}u-P_n u||_{L^2([-2^N,2^N]^d)} + ||\nabla u||_{L^2([-2^N,2^N]^d)} \right) ||\nabla v||_{L^2}.
\end{align*}
Taking the supremum over all $v$ which are supported in $[-2^N,2^N]^d$ and satisfy $||\nabla v||_{L^2}\leq 1$ gives
\begin{align}
\label{EstimateH-1}
\left|\left|u-\fint_{[-2^N,2^N]^d} u\right|\right|_{H^{-1}([-2^N,2^N]^d)} \lesssim \sum_{n=1}^N 2^n ||P_{n-1}u-P_n u||_{L^2([-2^N,2^N]^d)} + ||\nabla u||_{L^2([-2^N,2^N]^d)}.
\end{align}
Notice that for a cube $Q=[-2^n,2^n]$ and its subcube $[0,2^n]^d$ we may construct a vector field $g_{Q,1}$ supported in $\{|x|\leq 2^d 2^n\}$ satisfying
\begin{align*}
\fint_{[0,2^n]^d} u \,dx-\fint_{[-2^n,2^n]^d} u \,dx
=2^n \int g_{Q,1} \cdot \nabla u \,dx
\end{align*}
as well as $(\fint_{\{|x|\leq 2^d 2^n\}} |g_{Q,1}|^{2+1/d} )^{1/(2+1/d)}\lesssim (2^d 2^n)^{-d}$ (and similar vector fields for the other $2^d-1$ dyadic subcubes of $Q=[-2^n,2^n]$; we shall denote these vector fields by $g_{Q,2},\ldots,g_{Q,2^d}$); to see that the dependence on $n$ of this bound is the correct one, note that the vector field for $n\neq 1$ may be obtained from the vector field for $n=1$ by rescaling. Analogously, we may construct such vector fields $g_{Q,i}$ with analogous properties for any cube $Q$ with side length $2\cdot 2^n$ and its $2^d$ dyadic subcubes. Denoting by $\mathcal{Q}_n$ the set of cubes of side length $2\cdot 2^n$ obtained by decomposing the cube $[-2^N,2^N]$ into $2^{d(N-n)}$ cubes, this entails
\begin{align*}
||P_{n-1} u - P_n u||_{L^2([-2^N,2^N]^d)}^2
\leq
\sum_{Q\in \mathcal{Q}_n} \sum_{i=1}^{2^d} 2^{-d} |Q|
\left|2^n \int g_{Q,i} \cdot \nabla u \,dx\right|^2.
\end{align*}
Multiplying both sides with $2^{-d(N+1)}$, raising both sides to the $m/2$-th power, and using Jensen's inequality for the sum (note that we have $2^{-d(N+1)}\sum_{\mathcal{Q}_n} \sum_{i=1}^{2^d} 2^{-d} |Q|=1$), we deduce
\begin{align*}
\big(2^{-d(N+1)/2}||P_{n-1} u - P_n u||_{L^2([-2^N,2^N]^d)}\big)^{m}
\leq
2^{nm} \cdot 2^{-d(N+1)} \sum_{Q\in \mathcal{Q}_n} \sum_{i=1}^{2^d} 2^{-d} |Q|
\left|\int g_{Q,i} \cdot \nabla u \,dx\right|^{m}.
\end{align*}
Taking the expectation, inserting the assumption \eqref{IntLemma.AssumptionGradientAverages}, and raising both sides to the $1/m$-th power, we infer
\begin{align}
\label{EstimateOnLevel}
\mathbb{E}\Big[\big(2^{-d(N+1)/2}||P_{n-1} u - P_n u||_{L^2([-2^N,2^N]^d)}\big)^{m} \Big]^{1/m}
\lesssim K (2^n)^{1-\gamma}.
\end{align}
Together with \eqref{EstimateH-1} and the assumption \eqref{IntLemma.AssumptionGradient}, this yields the desired estimate \eqref{IntLemma.Result2} for the $H^{-1}$ norm.

To see the estimate on the $L^2$ norm, we use the fact that the terms in the decomposition
\begin{align*}
u-\fint_{[-2^N,2^N]^d} u=u-P_N u=(u-P_0 u)+\sum_{n=1}^N (P_{n-1} u -P_n u)
\end{align*}
are mutually orthogonal with respect to the $L^2$ scalar product. This yields
\begin{align*}
\left|\left|u-\fint_{[-2^N,2^N]^d} u\right|\right|_{L^2([-2^N,2^N]^d)}^2 = \sum_{n=1}^N ||P_{n-1}u-P_n u||_{L^2([-2^N,2^N]^d)}^2 + ||u-P_0 u||_{L^2([-2^N,2^N]^d)}^2.
\end{align*}
We now apply the Poincar\'e inequality to the last term to deduce
\begin{align}
\label{EstimateL2}
\left|\left|u-\fint_{[-2^N,2^N]^d} u\right|\right|_{L^2([-2^N,2^N]^d)}^2
\lesssim \sum_{n=1}^N ||P_{n-1}u-P_n u||_{L^2([-2^N,2^N]^d)}^2 + ||\nabla u||_{L^2([-2^N,2^N]^d)}^2.
\end{align}
Making again use of the bound \eqref{EstimateOnLevel} and the assumption \eqref{IntLemma.AssumptionGradient}, we infer \eqref{IntLemma.Result2}. Note that due to the squares present in the bound \eqref{EstimateL2} (as opposed to the bound \eqref{EstimateH-1}),  in the critical case $\gamma=1$ we now obtain a factor $\sqrt{N}\sim \sqrt{\log 2^N}$ in the estimate for $||u-\fint_{[-2^N,2^N]^d} u ||_{L^2([-2^N,2^N]^d)}$ (as opposed to a factor $N\sim \log 2^N$ in the bound for $||u-\fint_{[-2^N,2^N]^d} u ||_{H^{-1}([-2^N,2^N]^d)}$ for $\gamma=2$).
\end{proof}

\begin{proof}[Proof of Theorem \ref{SecondCorrectorEstimate}]
Theorem \ref{SecondCorrectorEstimate} is basically a consequence of Proposition \ref{MomentBoundsCorrectorAverages}, Proposition~\ref{proppsimoment}, and Lemma \ref{IntegrateLemma}: The combination of Proposition \ref{MomentBoundsCorrectorAverages} and Proposition~\ref{proppsimoment} with Lemma \ref{IntegrateLemma} (the lemma being applied with $K:=C(d,\lambda,\beta)m^{C(d,\lambda,\beta)}$ and $\gamma:=(1-\beta)d/2-1$) yields the estimate
\begin{align*}
&\mathbb{E}\left[
\left(\fint_{\{|x|\leq r\}} \Big|(\psi,\Psi)-\fint_{\{|x|\leq r\}}(\psi,\Psi)\Big|^2 \,dx\right)^{m/2} \right]^{1/m}
\\&
\leq
\begin{cases}
C_3 C(d,\lambda,\beta) m^{C(d,\lambda,\beta)} r^{2-(1-\beta)d/2} &\text{for }(1-\beta)d/2<2,
\\
C_3 C(d,\lambda,\beta) m^{C(d,\lambda,\beta)} \sqrt{\log(r+1)} &\text{for }(1-\beta)d/2=2,
\\
C_3 C(d,\lambda,\beta) m^{C(d,\lambda,\beta)} &\text{for }(1-\beta)d/2>2,
\end{cases}
\end{align*}
for any $m\geq 1$. By elementary theory of integration, this implies the desired bound with stretched exponential moments of the random constant $\mathcal{C}_r(a)$.
\end{proof}

\bibliographystyle{abbrv}
\bibliography{second_order_corrector}

\end{document}